\documentclass[amssymb,amsfonts,refcheck,11pt,verbatim,righttag]{amsart}
\pdfoutput=1 
\usepackage{amsfonts}

\usepackage{amsmath}
\usepackage{amssymb}
\usepackage{stmaryrd}
\usepackage{mathtools}
\usepackage{nccmath}
\usepackage{graphicx}
\usepackage{amsthm}
\usepackage{sidecap} 
\usepackage{xfrac} 

\usepackage[latin1]{inputenc} 
\usepackage[T1]{fontenc}      
\usepackage{lmodern}
\usepackage{bm}
\usepackage[french,english]{babel}

\usepackage{geometry}
\usepackage{enumitem}
\usepackage{float}

\usepackage{mathtools}
\usepackage{dsfont} 
\usepackage{hyperref}
\usepackage{url}
\hypersetup{colorlinks,citecolor=blue,filecolor=blue,linkcolor=blue,urlcolor=blue}
\usepackage{verbatim}
\usepackage{setspace}

\newtheorem{theorem}{Theorem}[section]
\newtheorem{proposition}[theorem]{Proposition}
\newtheorem{corollary}[theorem]{Corollary}
\newtheorem{lemma}[theorem]{Lemma}

\theoremstyle{definition}

\theoremstyle{remark}

\newtheorem{example}[theorem]{Example}

\allowdisplaybreaks

\setlist[itemize]{label=\textbullet}

\title{DIMENSIONS OF ``SELF-AFFINE SPONGES'' INVARIANT UNDER THE ACTION OF MULTIPLICATIVE INTEGERS}
\author{GUILHEM BRUNET}
\address{Laboratoire d'Analyse, G\'eom\'etrie et Applications, CNRS, UMR 7539, Universit\'e Sorbonne Paris Nord, CNRS, UMR 7539,  F-93430, Villetaneuse, France}
\email{brunet@math.univ-paris13.fr}
\begin{document}
\raggedbottom
\pagestyle{plain}

\begin{abstract}
Let $m_1 \geq m_2 \geq 2$ be integers. We consider subsets of the product symbolic sequence space $(\{0,\dots,m_1-1\} \times \{0,\dots,m_2-1\})^{\mathbb{N}^*}$ that are invariant under the action of the semigroup of multiplicative integers. These sets are defined following Kenyon, Peres and Solomyak and using a fixed integer $q \geq 2$. We compute the Hausdorff and Minkowski dimensions of the projection of these sets onto an affine grid of the unit square. The proof of our Hausdorff dimension formula proceeds via a variational principle over some class of Borel probability measures on the studied sets. This extends well-known results on self-affine Sierpi\'nski carpets. However, the combinatoric arguments we use in our proofs are more elaborate than in the self-similar case and involve a new parameter, namely $j = \left\lfloor \log_q \left( \frac{\log(m_1)}{\log(m_2)} \right) \right\rfloor$. We then generalize our results to the same subsets defined in dimension $d \geq 2$. There, the situation is even more delicate and our formulas involve a collection of $2d-3$ parameters.\newline
\end{abstract}

\maketitle
For the reader's convenience we summarize a list of commonly used symbols below :
\par\noindent\rule{155mm}{.1pt}\newline
\begin{tabular}{lll}
   $\mathcal{A}_i$ & Alphabet $\{0,\dots,m_i-1\}$ \\
   $\Sigma_{m_1,m_2}$ & Symbolic space $(\mathcal{A}_1 \times \mathcal{A}_2)^{\mathbb{N}^*}$\\
	 $q$ & Integer $\geq 2$\\
	 $\Omega$ & Closed subset of $\Sigma_{m_1,m_2}$\\
	 $X_\Omega$ & Closed subset of $\Sigma_{m_1,m_2}$ invariant under the action of multiplicative integers\\
   $\sigma$ & Standard shift map on $\Sigma_{m_1,m_2}$ \\
   $\gamma$ & $\gamma := \frac{\log(m_2)}{\log(m_1)}$\\
	 $(x,y)|_{J_i}$ & $(x,y)|_{J_i} := ((x_{q^{\ell} i},y_{q^{\ell} i}))_{\ell =0}^\infty$\\
	 $L$ & Map $n \in \mathbb{N}^* \mapsto \left\lceil \frac{n}{\gamma} \right\rceil$\\
	 $\mu$ & Borel probability measure on $\Omega$\\
	 $\mathbb{P}_\mu$ & Borel probability measure on $X_\Omega$, see Section \ref{21}\\
	 $\pi$ & Projection map of $\Sigma_{m_1,m_2}$ on the second coordinate\\
	 $\Omega_y$ & $\Omega_y := \Omega \cap \pi^{-1}(\{y\})$\\
   $\left[u\right]$ & Generalized cylinder on $\Sigma_{m_1,m_2}$, see Section \ref{21}\\
	 $\textup{Pref}_{p,\ell}(\Omega)$ & $(p \times \ell)$-sized prefixes of $\Omega$, see Section \ref{21}\\
	 $\alpha^1_k$ & $\alpha^1_k := \{\Omega \cap \left[u\right] :  u \in \textup{Pref}_{0,k}(\Omega)\}$\\

\end{tabular}
\par\noindent\rule{155mm}{.1pt}
\par\noindent\rule{155mm}{.1pt}\newline
\begin{tabular}{lll}
   $\alpha^2_k$ & $\alpha^2_k := \{\Omega \cap \left[u\right] :  u \in \textup{Pref}_{k,0}(\Omega)\}$\\
   $H^\mu_{m_2}$ & $\mu$-entropy of a finite partition with the base-$m_2$ logarithm\\
	 $j$ & The unique non-negative integer such that $q^j \leq \gamma^{-1} < q^{j+1}$\\
	 $\Omega_u$ & For $u = (x_1,y_1) \cdots (x_k,y_k) y_{k+1} \cdots y_{k+j} \in \text{Pref}_{k,j}(\Omega)$, $\Omega_u$ is the follower set of\\ 
	 $\ \ $ & $(x_1,y_1) \cdots (x_k,y_k)$ in $\Omega$ with $y_{k+1},\ldots,y_{k+j}$ being fixed \\
	 $\mu_u$ & The normalized measure induced by $\mu$ on $\Omega_u$\\
	 $\dim_e(\nu)$ & Entropy dimension of the measure $\nu$\\
	 $\nu^y$ & Disintegration of the measure $\nu$ with respect to $\pi$\\
	 $\Gamma_j(\Omega)$ & $j^{\text{th}}$ tree of prefixes of $\Omega$, see Section \ref{23}\\
	 $\Gamma_{u,j}(\Omega)$ & Tree of followers of $u$ in $\Gamma_j(\Omega)$, see Section \ref{23}\\
	 $t = t(u)$ & The unique vector defined on the set of vertices of $\Gamma_{u,j}(\Omega)$ satisfying equation \eqref{solution}\\ 
	 $t_\varnothing$ & See Section \ref{23}
\end{tabular}
\par\noindent\rule{155mm}{.1pt}\newline
\section{Introduction}

Let $m_1 \geq m_2 \geq 2$ and $q \geq 2$ be integers. Let $\Omega$ be a closed subset of $$\Sigma_{m_1,m_2} = (\mathcal{A}_1 \times \mathcal{A}_2)^{\mathbb{N}^*},$$ where $\mathcal{A}_1 = \{0,\dots,m_1-1\}$ and $\mathcal{A}_2 = \{0,\dots,m_2-1\}$. We can associate to $\Omega$ a closed subset of the torus $\mathbb{T}^2$ by considering $\psi(\Omega)$, where $\psi$ is the coding map defined as
$$
\psi : (x_k,y_k)_{k=1}^\infty \in \Sigma_{m_1,m_2} \longmapsto \left( \sum_{k=1}^\infty \frac{x_k}{m_1^k}, \sum_{k=1}^\infty \frac{y_k}{m_2^k} \right) \in \mathbb{T}^2.
$$
Let $\sigma$ be the standard shift map on $\Sigma_{m_1,m_2}$ and $\pi$ be the projection on the second coordinate. Closed subsets of $\Sigma_{m_1,m_2}$ that are $\sigma$-invariant are sent through $\psi$ to closed subsets of $\mathbb{T}^2$ that are invariant under the diagonal endomorphism of $\mathbb{T}^2$ ; $$(x,y) \in \mathbb{T}^2 \longmapsto (m_1 x, m_2 x).$$ Classical examples of such subsets are Sierpi\'nski carpets. Given $$\emptyset \neq A \subset \{0,\dots,m_1-1\} \times \{0,\dots,m_2-1\},$$ consider 
$$
\Omega = \{(x,y) = (x_k,y_k)_{k=1}^\infty \in \Sigma_{m_1,m_2} : \forall  k \geq 1, \ (x_k,y_k) \in A\}.
$$
Then $\psi(\Omega)$ is a Sierpi\'nski carpet. In this case, $\psi(\Omega)$ is the attractor of the iterated function system made of the contractions $f_{(i,j)} : (x,y) \in \mathbb{T}^2 \mapsto \left( \frac{x+i}{m_1}, \frac{y+j}{m_2} \right)$ with $(i,j) \in A$. When $m_1 = m_2 = m$, we obtain a self-similar fractal and it is well-known that $$\dim_H(\psi(\Omega)) = \dim_M(\psi(\Omega)) = \frac{\log(\# A)}{\log(m)},$$ where $\dim_H$ and $\dim_M$ stand for the Hausdorff and Minkowski (also called box-counting) dimensions respectively. See for example Chapter $2$ of \cite{ref3}. More generally, as proved in \cite{ref6}, if $\Omega$ is a closed shift-invariant subset of $\Sigma_{m,m}$ then we have
$$
\dim_H(\psi(\Omega)) = \dim_M(\psi(\Omega)) = \frac{h \left(\sigma|_\Omega\right)}{\log(m)},
$$  
where $h$ stands for the topological entropy. McMullen \cite{ref1} and Bedford \cite{ref4} independently computed the Hausdorff and Minkowski dimensions of general Sierpi\'nski carpets when $m_1 > m_2$, which we will assume from now on. Furthermore, the Hausdorff and Minkowski dimensions of Sierpi\'nski sponges - defined as the generalization of Sierpi\'nski carpets in all dimensions - were later computed in \cite{ref10}.

Let $$\gamma = \frac{\log(m_2)}{\log(m_1)}$$ and $$L : n \in \mathbb{N}^* \longmapsto \left\lceil \frac{n}{\gamma} \right\rceil.$$ We will need the following metric on $\Sigma_{m_1,m_2}$ : for $(x,y)$ and $(u,v)$ in $\Sigma_{m_1,m_2}$ let
\begin{align*}
\begin{split}
&d((x_k,y_k)_{k=1}^\infty, (u_k,v_k)_{k=1}^\infty) \\&= \max \left(m_1^{-\min \{k \geq 0 : (x_{k+1},y_{k+1}) \neq (u_{k+1},v_{k+1})\}}, \ m_1^{-\gamma \min \{k \geq 0 : y_{k+1} \neq v_{k+1} \}}\right).
\end{split}
\end{align*}
This metric allows us to consider ``quasi-squares'' as defined by McMullen when computing the dimensions of Sierpi\'nski carpets. It is easy to see that for $(x,y) \in \Sigma_{m_1,m_2}$ the balls centered at $(x,y)$ are
$$
B_n(x,y) = B_{m_1^{-n}}(x,y) = \{(u,v) \in \Sigma_{m_1,m_2} : u_k = x_k \ \forall 1 \leq k \leq n \ \text{and} \ v_k = y_k \ \forall 1 \leq k \leq L(n)\}.
$$
Using this metric on $\Sigma_{m_1,m_2}$ the Hausdorff and Minkowski dimensions of $\Omega$ are then equal to those of $\psi(\Omega)$. Thus from now on we will only work on the symbolic space. In this paper, our goal is to compute the Hausdorff and Minkowski dimensions of more general carpets that are not shift invariant. More precisely, given an arbitrary closed subset $\Omega$ of $\Sigma_{m_1,m_2}$ we consider
$$
X_\Omega = \{ (x_k,y_k)_{k=1}^\infty \in \Sigma_{m_1,m_2} : (x_{i q^\ell},y_{i q^\ell})_{\ell=0}^\infty \in \Omega \ \textup{for all} \ i, \ q \nmid i\}.
$$
Such sets were studied in \cite{ref2}, where the authors restricted their work to the one dimensional case : they computed the Hausdorff and Minkowski dimensions of sets defined by
$$
\left\{ (x_k)_{k=1}^\infty \in \left\{0,\dots,m-1\right\}^{\mathbb{N}^*} : (x_{i q^\ell})_{\ell=0}^\infty \in \Omega \ \textup{for all} \ i, \ q \nmid i\right\},
$$
where $\Omega$ is an arbitrary closed subset of $\left\{0,\dots,m-1\right\}^{\mathbb{N}^*}$. It is easily seen that this case covers the situation where $m_1 = m_2$ in our setting. Their interest in these sets was prompted by the computation of the Minkowski dimension of the "multiplicative golden mean shift"
$$
\left\{x = \sum_{k=1}^\infty \frac{x_k}{2^k} : x_k \in \{0,1\} \ \text{and} \ x_k x_{2k} = 0 \ \text{for all} \ k \geq 1\right\}
$$
done in \cite{ref7}. We aim to give formulas for $\dim_H(X_\Omega)$ and $\dim_M(X_\Omega)$ in the two-dimensional case, and then in all dimensions. Note that if $\Omega$ is shift-invariant, then $X_\Omega$ is invariant under the action of any integer $r \in \mathbb{N}^*$ $$(x_k,y_k)_{k=1}^\infty \longmapsto (x_{rk},y_{rk})_{k=1}^\infty.$$ For example, as in the case of dimension one we can consider subshifts of finite type on $\Sigma_{m_1,m_2}$. To do so, let $D = \{(0,0),(0,1), \dots ,(0,m_2-1),(1,0),(1,1),\dots,(1,m_2-1),\dots,(m_1-1,0),(m_1-1,1),\dots,(m_1-1,m_2-1)\}$ and let $A$ be an $m_1 m_2$ - sized square matrix indexed by $D \times D$ with entries in $\{0,1\}$. Then define 
$$
\Sigma_A = \{ (x_k,y_k)_{k=1}^\infty \in \Sigma_{m_1,m_2} : A((x_k,y_k),(x_{k+1},y_{k+1})) = 1, \ k \geq 1 \},
$$
and
$$
X_A = X_{\Sigma_A} = \{ (x_k,y_k)_{k=1}^\infty \in \Sigma_{m_1,m_2} : A((x_k,y_k),(x_{qk},y_{qk})) = 1, \ k \geq 1 \}.
$$\newline

\begin{figure}[H]
\centering
\includegraphics[scale=0.39,trim={14cm 2cm 14cm 3cm},clip]{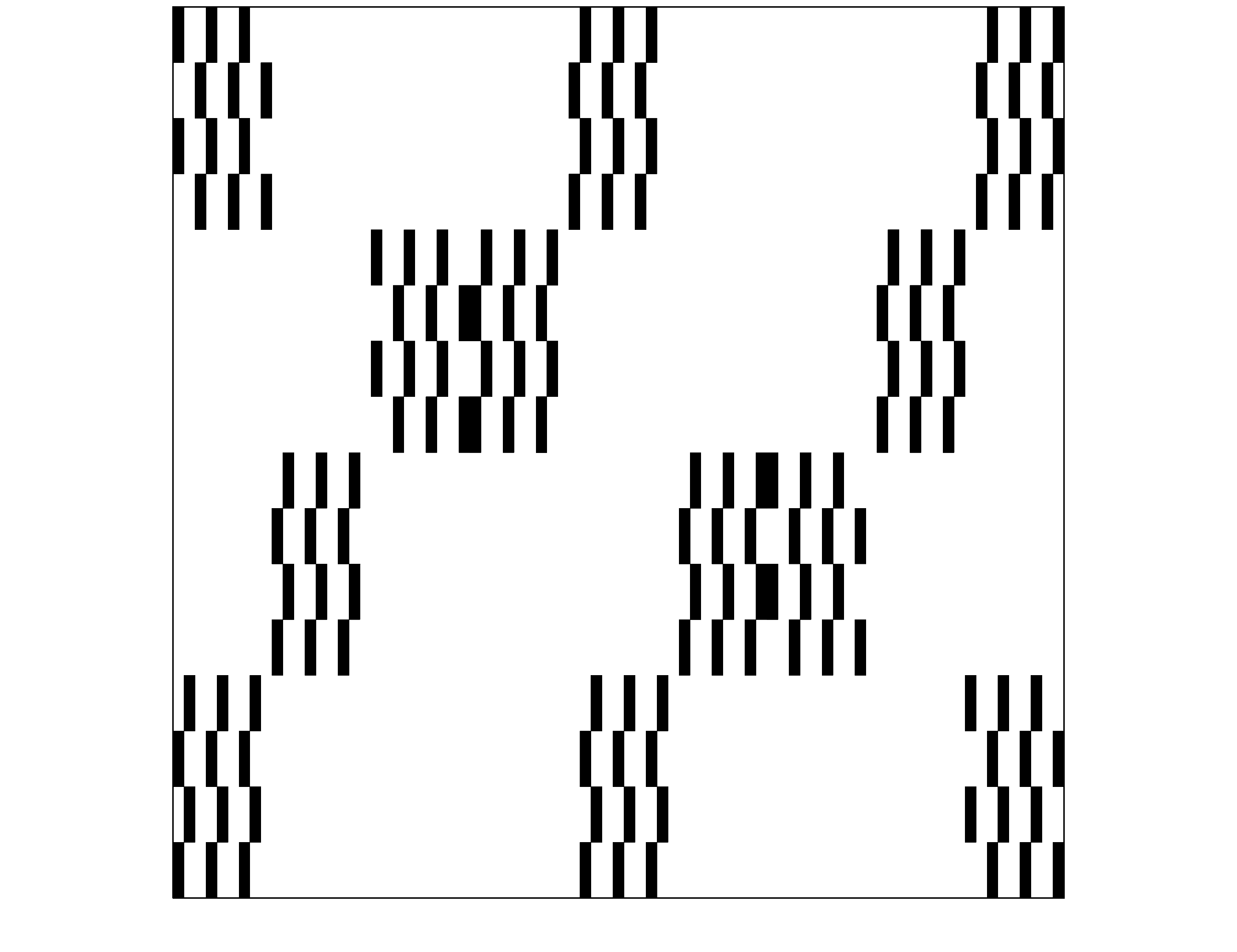}
\caption{Approximation of order $4$ of the set $X_A$ for $m_1 = 3$, $m_2 = 2$, $q=2$ and $A$ a circulant matrix whose first row is $(1,0,0,1,0,0)$.}
\end{figure}

\begin{figure}[H]
\centering
\includegraphics[scale=0.10]{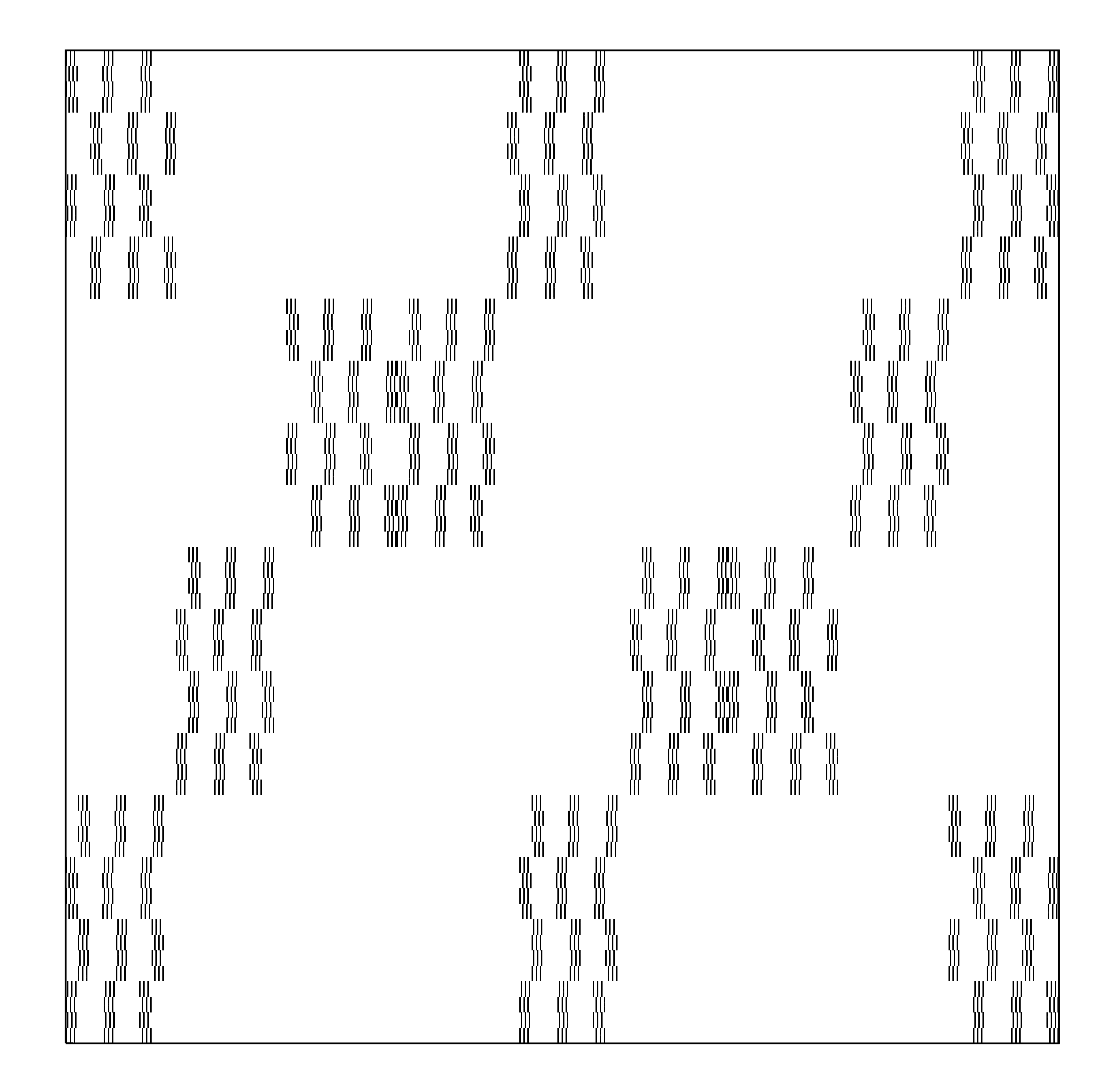}
\caption{Approximation of order $6$ of the set $X_A$ for $m_1 = 3$, $m_2 = 2$, $q=2$ and $A$ a circulant matrix whose first row is $(1,0,0,1,0,0)$.}
\end{figure}

Note that further generalizations of the sets considered in \cite{ref2} were studied in \cite{ref8}, in the one-dimensional case as well.\newline 

The paper is organized as follows. In Section \ref{2}, we focus on the two-dimensional situation. We first introduce in Subsection \ref{21} a particular class of measures on $X_\Omega$. We show that these measures are exact dimensional and we compute their Hausdorff dimensions. This class of measures is the same as that considered in \cite{ref2}, but in our case the parameter $j = \left\lfloor \log_q \left( \frac{\log(m_1)}{\log(m_2)} \right) \right\rfloor$ comes into play when studying their local dimension. Indeed, this parameter plays a crucial role in the definition of generalized cylinders whose masses are used to study the mass of balls under the metric $d$. In Subsection \ref{22}, out of curiosity, we study under which condition the Ledrappier--Young formula (where the entropies of invariant measures are replaced by their entropy dimensions) can hold for these measures, which are not shift-invariant in general.\newline

In Subsections \ref{23}, \ref{24} and \ref{25}, we compute the Hausdorff and Minkowski dimensions of $X_\Omega$, using a variational principle over the class of measures we studied earlier. We show that there exists a unique Borel probability measure which allows us to bound $\dim_H(X_\Omega)$ both from below and from above.\newline 

Then, in Section \ref{3}, we extend our results to the general multidimensional case. The combinatorics involved there become significantly more complex, as the study of the local dimension of the measures of interest invokes some generalized cylinders which depend in a subtle way on a collection of $2d-3$ parameters.

\section{The two-dimensional case}\label{2}

\subsection{The measures $\mathbb{P}_\mu$ and their dimensions}\label{21}

Throughout the paper we will use the notation $\llbracket m,n \rrbracket = \left\{m, \dots, n\right\}$ if $m \leq n$ are integers.\newline

To compute $\dim_H(X_\Omega)$, we will use the classical strategy of stating a variational principle over a certain class of Borel probability measures $\mathbb{P}_\mu$ on $X_\Omega$ defined below, i.e we will show that
$$
\dim_H(X_\Omega) = \max_{\mathbb{P}_\mu} \dim_H(\mathbb{P}_\mu).
$$
To do so, we will use the following classical facts (for a proof, see \cite[Proposition 2.3]{ref3}) :\newline

\begin{theorem} \label{calme}
Let $\mu$ be a finite Borel measure on $\Sigma_{m_1,m_2}$ and let $A \subset \Sigma_{m_1,m_2}$ such that $\mu(A) > 0$.
\begin{itemize}
\item If $\liminf_{n \rightarrow \infty} -\frac{\log_{m_1}(\mu(B_n(x)))}{n} \geq D$ for $\mu$-almost all $x$, then $\dim_H(\mu) \geq D$.
\item If $\liminf_{n \rightarrow \infty} -\frac{\log_{m_1}(\mu(B_n(x)))}{n} \leq D$ for $\mu$-almost all $x$, then $\dim_H(\mu) \leq D$.
\item If $\liminf_{n \rightarrow \infty} -\frac{\log_{m_1}(\mu(B_n(x)))}{n} \leq D$ for all $x \in A$, then $\dim_H(A) \leq D$.\newline
\end{itemize}
\end{theorem} 

For $p,\ell \in \mathbb{N}$ and $u \in (\{0,\dots,m_1-1\} \times \{0,\dots,m_2-1\})^p \times \{0,\dots,m_2-1\}^\ell$, define the generalized cylinder
$$
\left[u\right] = \{(x,y) \in \Sigma_{m_1,m_2} : ((x,y)|_p, \pi(\sigma^p((x,y))|_\ell)) = u\},
$$
where $(x,y)|_p = (x_1,y_1) \cdots (x_p,y_p)$ and $\pi((x,y)|_p) = y|_p$, and set  $$\textup{Pref}_{p,\ell}(\Omega) = \{u \in (\{0,\dots,m_1-1\} \times \{0,\dots,m_2-1\})^p \times \{0,\dots,m_2-1\}^\ell : \Omega \cap \left[u\right] \neq \varnothing\}.$$ For $(x,y) \in \Sigma_{m_1,m_2}$, $n \geq 1$ and $i$ an integer such that $q \nmid i$, we define $$(x,y)|_{J_i^n} = (x_i,y_i) (x_{qi},y_{qi}) \cdots (x_{q^r i},y_{q^r i})$$ if $q^r i \leq n < q^{r+1} i$. Let $\mu$ be a Borel probability measure on $\Omega$. Following \cite{ref2} we define $\mathbb{P}_\mu$ on the semi-algebra of cylinder sets of $\Sigma_{m_1,m_2}$ by
$$
\mathbb{P}_\mu(\left[(x,y)|_n\right]) = \prod_{\substack{i \leq n \\ q \nmid i}} \mu \left(\left[(x,y)|_{J_i^n}\right] \right).
$$
This is a well defined pre-measure. Indeed it is easy to see that $\mathbb{P}_\mu(\left[(k,l)\right]) = \mu(\left[(k,l)\right])$ for $(k,l) \in \mathcal{A}_1 \times \mathcal{A}_2$, and for $n+1 = q^r i$ with $q \nmid i$,
$$
\frac{\mathbb{P}_\mu( \left[(x_1,y_1) \cdots (x_n,y_n) (x_{n+1},y_{n+1}) \right])}{\mathbb{P}_\mu( \left[(x_1,y_1) \cdots (x_n,y_n) \right] )} = \frac{\mu(\left[(x_i,y_1) (x_{qi},y_{qi}) \cdots (x_{q^r i},y_{q^r i}) \right])}{\mu \left(\left[(x_i,y_1) (x_{qi},y_{qi}) \cdots (x_{q^{r-1} i},y_{q^{r-1} i}) \right] \right)},
$$
whence 
$$
\mathbb{P}_\mu( \left[(x_1,y_1) \cdots (x_n,y_n) \right]) = \sum_{(i,j) \in \mathcal{A}_1 \times \mathcal{A}_2} \mathbb{P}_\mu( \left[(x_1,y_1) \cdots (x_n,y_n) (i,j) \right]).
$$
Denote also by $\mathbb{P}_\mu$ the extension of $\mathbb{P}_\mu$ to a Borel probability measure on $\left(\Sigma_{m_1,m_2},\mathcal{B} \left(\Sigma_{m_1,m_2} \right) \right)$. By construction, $\mathbb{P}_\mu$ is supported on $X_\Omega$, since $\Omega$ is a closed subset of $\Sigma_{m_1,m_2}$ and hence
$$
\Omega =  \bigcap_{k=1}^\infty \bigcup_{u \in \textup{Pref}_{k,0}(\Omega)} \left[u\right].
$$\newline

Let us now introduce some more notations. For all $k \geq 1$ we consider the finite partitions of $\Omega$ defined by
$$
\alpha^1_k =  \{\Omega \cap \left[u\right] : u \in \textup{Pref}_{0,k}(\Omega)\}
$$
and
$$
\alpha^2_k = \{\Omega \cap \left[u\right] : u \in \textup{Pref}_{k,0}(\Omega)\}.
$$
For a Borel probability measure $\mu$ on $\Omega$ and a finite measurable partition $\mathcal{P}$ on $\Omega$, denote by $H^\mu_{m_2}(\mathcal{P})$ the $\mu$-entropy of the partition, with the base-$m_2$ logarithm :
$$
H^\mu_{m_2}(\mathcal{P}) = - \sum_{C \in \mathcal{P}} \mu(C) \log_{m_2} \mu(C). 
$$\newline
Let $j$ be the unique non-negative integer such that $$q^j \leq \frac{1}{\gamma} = \frac{\log(m_1)}{\log(m_2)} < q^{j+1}.$$ Note that for all $n \geq 1$ large enough we have $$q^j n \leq L(n) < q^{j+1} n.$$ 
\begin{theorem} \label{test2} Let $\mu$ be a Borel probability measure on $\Omega$. Then $\mathbb{P}_\mu$ is exact dimensional and we have
\begin{align*}
\begin{split}
\dim_H(\mathbb{P}_\mu) = (q-1)^2 \sum_{p=1}^j \frac{H^\mu_{m_2}(\alpha^1_p)}{q^{p+1}} &+ (q-1)(q^{j+1} \gamma - 1) \sum_{p=j+1}^\infty \frac{H^\mu_{m_2}(\alpha^2_{p-j} \lor \alpha^1_p)}{q^{p+1}} \\ &+ (q-1)(1-q^j \gamma) \sum_{p=j+1}^\infty \frac{H^\mu_{m_2}(\alpha^2_{p-j-1} \lor \alpha^1_p)}{q^p}.
\end{split}
\end{align*}
\end{theorem}

\begin{proof}
Our method is inspired by the calculation of $\dim_H(\mathbb{P}_\mu)$ in \cite{ref2}. The strategy of the proof is the same, nevertheless the computations will be more involved, due to the fact that the $\mathbb{P}_\mu$-mass of a ball for the metric $d$ is a product of $\mu$-masses of generalized cylinders rather than standard ones as in \cite{ref2}.\newline \newline

 Let $\ell \geq j+1$. We will first show that for $\mathbb{P}_\mu$-almost all $(x,y) \in X_\Omega$ we have
\begin{align*}
\begin{split}
\liminf_{n \rightarrow \infty} \frac{-\log_{m_1}(\mathbb{P}_\mu(B_{n}(x,y)))}{n} &\geq (q-1)^2 \sum_{p=1}^j \frac{H^\mu_{m_2}(\alpha^1_p)}{q^{p+1}} \\ &+ (q-1)(q^{j+1} \gamma - 1) \sum_{p=j+1}^\ell \frac{H^\mu_{m_2}(\alpha^2_{p-j} \lor \alpha^1_p)}{q^{p+1}} \\ &+ (q-1)(1-q^j \gamma) \sum_{p=j+1}^\ell \frac{H^\mu_{m_2}(\alpha^2_{p-j-1} \lor \alpha^1_p)}{q^p},
\end{split}
\end{align*}

and 

\begin{align*}
\begin{split}
\limsup_{n \rightarrow \infty} \frac{-\log_{m_1}(\mathbb{P}_\mu(B_{n}(x,y)))}{n} &\leq  (q-1)^2 \sum_{p=1}^j \frac{H^\mu_{m_2}(\alpha^1_p)}{q^{p+1}} \\ &+ (q-1)(q^{j+1} \gamma - 1) \sum_{p=j+1}^\ell \frac{H^\mu_{m_2}(\alpha^2_{p-j} \lor \alpha^1_p)}{q^{p+1}} \\ &+ (q-1)(1-q^j \gamma) \sum_{p=j+1}^\ell \frac{H^\mu_{m_2}(\alpha^2_{p-j-1} \lor \alpha^1_p)}{q^k} \\ &+ \frac{(\ell+1) \log_{m_2}(m_1 m_2)}{q^\ell}.
\end{split}
\end{align*}
Letting $\ell \rightarrow \infty$ will yield the desired equality (cf. Theorem \ref{calme}). To check these, we can restrict ourselves to $n = q^\ell r$, $r \in \mathbb{N}$. Indeed if $q^\ell r \leq n < q^\ell (r+1)$ then
$$
\frac{-\log_{m_1}(\mathbb{P}_\mu(B_n(x,y)))}{n} \geq \frac{-\log_{m_1}(\mathbb{P}_\mu(B_{q^\ell r}(x,y)))}{q^\ell (r+1)} \geq \frac{r}{r+1} \frac{-\log_{m_1}(\mathbb{P}_\mu(B_{q^\ell r}(x,y)))}{q^\ell r},
$$
which gives

$$
\liminf_{n \rightarrow \infty} \frac{-\log_{m_1}(\mathbb{P}_\mu(B_{n}(x,y)))}{n} = \liminf_{r \rightarrow \infty} \frac{-\log_{m_1}(\mathbb{P}_\mu(B_{q^\ell r}(x,y)))}{q^\ell r}.
$$
The $\limsup$ is dealt with similarly.\newline

As proved in \cite{ref2} we have 

$$
\lim_{n\rightarrow \infty} \frac{-\log_{m_1} \left(\mathbb{P}_\mu \left(\left[(x,y)|_n\right]\right)\right)}{n} = (q-1)^2 \sum_{p=1}^\infty \frac{H^\mu_{m_1}(\alpha^2_p)}{q^{p+1}} \ \text{for} \ \mathbb{P}_\mu-\text{almost all} \ (x,y) \in X_\Omega.
$$
Note that 

$$
\mathbb{P}_\mu(B_n(x,y)) = \sum_{x'_{n+1},\dots,x'_{L(n)}} \mathbb{P}_\mu \left( \left[(x_1,y_1)\cdots(x_n,y_n)(x'_{n+1},y_{n+1})\cdots (x'_{L(n)},y_{L(n)})\right] \right),
$$
the sum being taken over all $x'_{n+1},\ldots,x'_{L(n)}$ such that $$\left[(x_1,y_1)\cdots(x_n,y_n)(x'_{n+1},y_{n+1})\cdots (x'_{L(n)},y_{L(n)})\right] \cap X_\Omega \neq \emptyset.$$
Let $$i \in \left]\frac{L(n)}{q^\ell},L(n)\right] = \bigsqcup_{p=1}^\ell \left]\frac{L(n)}{q^{p}},\frac{L(n)}{q^{p-1}}\right]$$ such that $q \nmid i$. Note that if $i \in \left]\frac{L(n)}{q^{p}},\frac{L(n)}{q^{p-1}}\right]$ then the word $(x,y)|_{J_i^{L(n)}}$ is of length $p$. Recall that $j$ is defined by $q^j \leq \frac{1}{\gamma} < q^{j+1}$. Suppose $j \geq 1$. If $1 \leq p \leq j$ then $\frac{L(n)}{q^p} \geq n$, so
$$
\left]\frac{L(n)}{q^p},\frac{L(n)}{q^{p-1}} \right] \subset \ ]n,L(n)].
$$
If $j+1 \leq p \leq \ell$ and $\ell$ is large enough, then $\frac{n}{q^{p-j-1}} \in \left]\frac{L(n)}{q^p},\frac{L(n)}{q^{p-1}}\right]$, thus we can partition 
$$
\left]\frac{L(n)}{q^p},\frac{L(n)}{q^{p-1}}\right] = \left]\frac{L(n)}{q^p},\frac{n}{q^{p-j-1}}\right] \bigsqcup \left]\frac{n}{q^{p-j-1}}, \frac{L(n)}{q^{p-1}}\right].
$$
In the case where $i \in \left]\frac{n}{q^{p-j-1}}, \frac{L(n)}{q^{p-1}}\right]$ we have
$$
q^{p-j-2} i \leq n < q^{p-j-1} i \leq q^{p-1} i \leq L(n) < q^p i,
$$
and if $i \in \left]\frac{L(n)}{q^p},\frac{n}{q^{p-j-1}}\right]$ then
$$
q^{p-j-1} i \leq n < q^{p-j} i \leq q^{p-1} i \leq L(n) < q^p i.$$ If $j=0$ then 
$$
i \in \left]\frac{n}{q^{p-1}}, \frac{L(n)}{q^{p-1}}\right] \Longrightarrow q^{p-2} i \leq n < q^{p-1} i \leq L(n) < q^p i 
$$

$$
i \in \left]\frac{L(n)}{q^p},\frac{n}{q^{p-1}}\right] \Longrightarrow q^{p-1} i \leq n \leq L(n) < q^p i.
$$
Thus for any $j$ we have

\begin{align*}
\begin{split}
&\mathbb{P}_\mu(B_n(x,y)) = \bigg[ \prod_{p=1}^j \prod_{\substack{i \in \left] \frac{L(n)}{q^p},\frac{L(n)}{q^{p-1}} \right] \\ q \nmid i}} \mu \left( \left[y_i \cdots y_{q^{p-1} i} \right] \right) \bigg] \\ &\boldsymbol{\cdot} \bigg[ \prod_{p=j+1}^\ell \bigg( \prod_{\substack{i \in \left]\frac{n}{q^{p-j-1}},\frac{L(n)}{q^{p-1}} \right] \\ q \nmid i}} \mu \left( \left[(x_i,y_i) \cdots (x_{q^{p-j-2} i}, y_{q^{p-j-2} i}) y_{q^{p-j-1} i} \cdots y_{q^{p-1} i} \right] \right) \bigg) \\ &\boldsymbol{\cdot} \bigg(\prod_{\substack{i \in \left]\frac{L(n)}{q^{p}},\frac{n}{q^{p-j-1}} \right] \\ q \nmid i}} \mu \left( \left[(x_i,y_i) \cdots (x_{q^{p-j-1} i}, y_{q^{p-j-1} i}) y_{q^{p-j} i} \cdots y_{q^{p-1} i} \right] \right) \bigg) \bigg] \boldsymbol{\cdot} D_n(x,y),
\end{split}
\end{align*}
with $D_n(x,y)$ being the product of the remaining quotients (words beginning with $(x_i,y_i)$ with $i \leq \frac{L(n)}{q^\ell}$). Here we used the notion of generalized cylinders we defined earlier :
\begin{align*}
\begin{split}
&\mu \left( \left[y_i \cdots y_{q^{p-1} i} \right] \right) = \sum_{x'_i,\ldots,x'_{q^{p-1} i}} \mu \left( \left[(x'_i,y_i) \cdots (x'_{q^{p-1} i},y_{q^{p-1} i}) \right] \right),\\&\mu \left( \left[(x_i,y_i) \cdots (x_{q^{p-j-2} i}, y_{q^{p-j-2} i}) y_{q^{p-j-1} i} \cdots y_{q^{p-1} i} \right] \right) \\&= \sum_{x'_{q^{p-j-1} i},\ldots,x'_{q^{p-1} i}} \mu \left( \left[(x_i,y_i) \cdots (x_{q^{p-j-2} i}, y_{q^{p-j-2} i}) (x'_{q^{p-j-1} i}, y_{q^{p-j-1} i}) \cdots (x'_{q^{p-1} i},y_{q^{p-1} i}) \right] \right),\\&\mu \left( \left[(x_i,y_i) \cdots (x_{q^{p-j-1} i}, y_{q^{p-j-1} i}) y_{q^{p-j} i} \cdots y_{q^{p-1} i} \right] \right) \\&= \sum_{x'_{q^{p-j} i},\ldots,x'_{q^{p-1} i}} \mu \left( \left[(x_i,y_i) \cdots (x_{q^{p-j-1} i}, y_{q^{p-j-1} i}) (x'_{q^{p-j} i}, y_{q^{p-j} i}) \cdots (x'_{q^{p-1} i},y_{q^{p-1} i}) \right] \right),
\end{split}
\end{align*}
the sums being taken over the cylinders that intersect $\Omega$. If $(u_n),(v_n) \in \left(\mathbb{R}^*\right)^{\mathbb{N}^*}$, we say that $u_n \sim v_n$ if $\frac{u_n}{v_n} \rightarrow 1$ as $n \rightarrow \infty$. Here we have

\begin{equation*}
\begin{aligned}
\# \left\{i \in \left]\frac{L(n)}{q^{p}},\frac{L(n)}{q^{p-1}}\right] : q \nmid i \right\}           & \sim \frac{(q-1)^2 n}{\gamma q^{p+1}},       \\
\# \left\{i \in \left]\frac{n}{q^{p-j-1}},\frac{L(n)}{q^{p-1}}\right] : q \nmid i \right\}    & \sim \frac{n (q-1)(1-q^j \gamma)}{\gamma q^p},    \\
\# \left\{ i \in \left] \frac{L(n)}{q^p},\frac{n}{q^{p-j-1}} \right] : q \nmid i \right\}       & \sim \frac{n (q-1) (q^{j+1} \gamma - 1)}{\gamma q^{p+1}}.
\end{aligned}
    \end{equation*}
Note that for $i \in \left]\frac{n}{q^{p-j-1}},\frac{L(n)}{q^{p-1}} \right]$, $q \nmid i$ the random variables $$Y_{i,n,p} : (x,y) \in X_\Omega \longmapsto - \log_{m_1} \left(\mu \left( \left[(x_i,y_i) \cdots (x_{q^{p-j-2} i}, y_{q^{p-j-2} i}) y_{q^{p-j-1} i} \cdots y_{q^{p-1} i} \right] \right) \right)$$ are i.i.d and uniformly bounded, with expectation being $H^\mu_{m_1}(\alpha^2_{p-j-1} \lor \alpha^1_p)$. Fixing $j+1 \leq p \leq l$ and letting $n = q^\ell r, r \rightarrow \infty$, we can use Lemma \ref{proba} to get that for $\mathbb{P}_\mu$-almost all $(x,y) \in X_\Omega$

$$
\frac{\gamma q^p}{n (q-1)(1-q^j \gamma)} \sum_{\substack{i \in \left]\frac{n}{q^{p-j-1}},\frac{L(n)}{q^{p-1}} \right] \\ q \nmid i}} Y_{i,n,p}(x,y) \underset{r \rightarrow \infty}{\longrightarrow} H^\mu_{m_1}(\alpha^2_{p-j-1} \lor \alpha^1_p).
$$
Thus

\begin{align*}
\begin{split}
&\sum_{p=j+1}^ \ell\frac{(q-1)(1-q^j \gamma)}{\gamma q^p} \sum_{\substack{i \in \left]\frac{n}{q^{p-j-1}},\frac{L(n)}{q^{p-1}} \right] \\ q \nmid i}} \frac{\gamma q^p Y_{i,n,p}(x,y)}{n (q-1)(1-q^j \gamma)} \\&\underset{r \rightarrow \infty}{\longrightarrow} (q-1)(1-q^j \gamma) \sum_{p=j+1}^\ell \frac{H^\mu_{m_2}(\alpha^2_{p-j-1} \lor \alpha^1_p)}{q^p}.
\end{split}
\end{align*}
Similarly if we define
$$
Z_{i,n,p} : (x,y) \longmapsto - \log_{m_1}\left(\mu \left( \left[y_i \cdots y_{q^{p-1} i} \right] \right) \right),
$$
whose expectation is $H^\mu_{m_1}(\alpha^1_p)$, for $\mathbb{P}_\mu$-almost all $(x,y) \in X_\Omega$ we have

$$
\frac{\gamma q^{p+1}}{n (q-1)^2} \sum_{\substack{i \in \left] \frac{L(n)}{q^p},\frac{L(n)}{q^{p-1}} \right] \\ q \nmid i}} Z_{i,n,p}(x,y) \underset{r \rightarrow \infty}{\longrightarrow} H^\mu_{m_1}(\alpha^1_p),
$$
hence

$$
\sum_{p=1}^j \frac{(q-1)^2}{\gamma q^{p+1}} \sum_{\substack{i \in \left] \frac{L(n)}{q^p},\frac{L(n)}{q^{p-1}} \right] \\ q \nmid i}} \frac{\gamma q^{p+1} Z_{i,n,p}(x,y)}{n (q-1)^2} \underset{r \rightarrow \infty}{\longrightarrow} (q-1)^2 \sum_{p=1}^j \frac{H^\mu_{m_2}(\alpha^1_p)}{q^{p+1}}.
$$
The third term is treated in similar manner. We have thus proved the first inequality. Now it remains to prove the second inequality using $D_n(x,y)$. It is easily seen that there exists $C \geq 0$ such that for all $b > a > 0$ 
$$
\left|\# \{i \in \mathbb{N} \cap \left]a,b\right] : q \nmid i\} - \frac{q-1}{q} (b-a) \right| \leq C.
$$
Thus the number of letters in $\mathcal{A}_1 \times \mathcal{A}_2$ appearing in the words of the developed $D_n(x,y)$ is 

\begin{align} \label{nombre}
\begin{split}
d_n &:= L(n) - \sum_{p=1}^\ell \# \left\{i \in \mathbb{N} \cap \left]\frac{L(n)}{q^p},\frac{L(n)}{q^{p-1}}\right] : q \nmid i\right\} p \\&\leq L(n) - \sum_{p=1}^\ell \frac{(q-1)^2 L(n) p}{q^{p+1}} + \frac{\ell(\ell+1)}{2} C \\ &= \frac{L(n)}{q^\ell} \left[(\ell+1)-\frac{\ell}{q} \right] + \frac{\ell(\ell+1)}{2} C \\ &\leq \frac{(\ell+1) L(n)}{q^\ell} + \frac{\ell(\ell+1)}{2} C.
\end{split}
\end{align}
On the other hand $$d_n \geq L(n) - \sum_{p=1}^\ell \frac{(q-1)^2 L(n) p}{q^{p+1}} - \frac{\ell(\ell+1)}{2} C \geq r \left[(\ell+1)-\frac{\ell}{q} \right] - \frac{\ell(\ell+1)}{2} C,$$ so $\sum_{r=1}^\infty 2^{-d_{q^\ell r}} < +\infty$. Define
$$
S_n = \{(x,y) \in X_\Omega : D_n(x,y) \leq (2 m_1 m_2)^{-d_n} \}.
$$
Clearly $\mathbb{P}_\mu(S_n) \leq 2^{-d_n}$, so $\mathbb{P}_\mu \left(\bigcap_{N \geq 1} \bigcup_{r=N}^\infty S_{q^\ell r} \right) = 0$, using Borel-Cantelli lemma. Hence for $\mathbb{P}_\mu$-almost all $(x,y) \in X_\Omega$ there exists $N(x,y)$ such that $(x,y) \notin S_n$ for all $n=q^\ell r \geq N(x,y)$. For such $(x,y)$ and $n \geq N(x,y)$, using \eqref{nombre}, we have
\begin{align*}
\begin{split}
\frac{-\log_{m_1}(D_n(x,y))}{n} &\leq \frac{d_n \log_{m_1}(2 m_1 m_2)}{n} \\&\leq \frac{(\ell+1) L(n) \log_{m_1}(2 m_1 m_2)}{n q^\ell} + \frac{\ell(\ell+1) \log_{m_1}(2 m_1 m_2)}{2n}.
\end{split}
\end{align*}
So

$$
\limsup_{r \rightarrow \infty} \frac{-\log_{m_1}(D_{q^\ell r}(x,y))}{q^\ell r} \leq \frac{(\ell+1) \log_{m_2}(2 m_1 m_2)}{q^\ell}.
$$
Finally for such $(x,y)$ we get the second desired inequality.

\end{proof}

\subsection{Study of the validity of the Ledrappier-Young formula}\label{22}

Here we will discuss the validity of the Ledrappier-Young formula in our context. Recall that for a shift-invariant ergodic measure $\mu$ on $\Sigma_{m_1,m_2}$, the Ledrappier-Young formula is (see \cite[Lemma 3.1]{ref10} for a proof)
$$
\dim_H(\mu) = \frac{1}{\log(m_1)} h_\mu(\sigma) + \left(\frac{1}{\log(m_2)}-\frac{1}{\log(m_1)}\right) h_{\pi_* \mu}(\tilde{\sigma}),
$$ 
where $\tilde{\sigma}$ is the standard shift map on $\Sigma_{m_2}$, $\pi$ is the projection on the second coordinate and $h_\mu(\sigma)$ is the entropy of $\mu$ with respect to $\sigma$. This rewrites as 
\begin{equation} \label{LY}
\dim_H(\mu) = \frac{1}{\log(m_1)} \dim_e(\mu) + \left(\frac{1}{\log(m_2)}-\frac{1}{\log(m_1)}\right) \dim_e(\pi_* \mu),
\end{equation}
where for any Borel probability measure $\nu$ on $\Sigma_{m_1,m_2}$, $\dim_e(\nu)$ denotes, whenever it exists, its entropy dimension defined by
$$
\dim_e(\nu) = \lim_{n\rightarrow \infty} -\frac{1}{n} \sum_{u \in (\mathcal{A}_1 \times \mathcal{A}_2)^n} \nu(\left[u\right]) \log \left(\nu \left(\left[u\right]\right)\right),
$$
and where $\dim_e(\pi_*\nu)$ is defined similarly. We will show that this fails to hold for $\mathbb{P}_\mu$ in general. This is expected since $\mathbb{P}_\mu$ is not shift-invariant in general. However, we will give a sufficient condition on $\mu$ for $\mathbb{P}_\mu$ to satisfy \eqref{LY}.\newline 

\hspace{0.01mm} Let $(\nu^y)_{y \in \pi \left(\Sigma_{m_1,m_2}\right)}$ be the $\pi_* \nu$-almost everywhere uniquely determined disintegration of the Borel probability measure $\nu$ on $\Sigma_{m_1,m_2}$ with respect to $\pi$. Each $\nu^y$ is a Borel probability measure on $\Sigma_{m_1,m_2}$ supported on $\pi^{-1}(\{y\})$, which can be computed using the formula
$$
\nu^y(\left[x|_n\right] \times \{y\}) = \lim_{p \rightarrow \infty} \frac{\nu \left(\left[(x_1,y_1) \cdots (x_n,y_n) y_{n+1} \cdots y_p\right] \right)}{\pi_* \nu \left(\left[y_1 \cdots y_p \right] \right)} \ \text{for} \ \pi_* \nu \text{-almost all} \ y \in \pi(\Sigma_{m_1,m_2}).
$$
For some basics on the notion of disintegrated measure we advise \cite{ref11} to the reader.\

\begin{proposition}
Let $\mu$ be a Borel probability measure on $\Omega$. Then $\pi_*(\mathbb{P}_\mu)$ is exact dimensional. Moreover $\mathbb{P}_\mu^y$ is exact dimensional for $\pi_*(\mathbb{P}_\mu) \text{-almost all} \ y \in \pi(X_\Omega)$, and we have $$\underset{y \sim \pi_*(\mathbb{P}_\mu)}{\emph{essinf}} \dim_e(\mathbb{P}_\mu^y) = \underset{y \sim \pi_*(\mathbb{P}_\mu)}{\emph{esssup}} \dim_e(\mathbb{P}_\mu^y).$$ Finally 
$$
\dim_e(\pi_*(\mathbb{P}_\mu)) + \underset{y \sim \pi_*(\mathbb{P}_\mu)}{\emph{essinf}} \dim_e(\mathbb{P}_\mu^y) \leq \dim_e(\mathbb{P}_\mu),
$$
with equality if and only if for all $p \geq 1$, for all $I \in \alpha^2_p$, the map $y \in \pi(I) \mapsto \mu^y(I)$ is $\pi_* \mu$-almost surely constant.

\end{proposition}

\begin{proof}

First note that for $(x,y) \in \Sigma_{m_1,m_2}$
$$
\pi_*(\mathbb{P}_\mu) (\left[y_1 \cdots y_n \right]) = \sum_{x_1,\ldots,x_n} \prod_{\substack{i \leq n \\ q \nmid i}} \mu \left(\left[(x,y)|_{J_i^n}\right] \right) = \prod_{\substack{i \leq n \\ q \nmid i}} \sum_{x_i, \ldots, x_{q^r i}} \mu \left( \left[(x,y)|_{J_i^n}\right] \right) = \mathbb{P}_{\pi_* \mu} (\left[y_1 \cdots y_n \right]).
$$
Thus $\pi_*(\mathbb{P}_\mu)$ is a Borel probability measure supported on $\pi(X_\Omega) = X_{\pi(\Omega)}$, which is equal to $\mathbb{P}_{\pi_* \mu}$. Thus, using the one-dimensional case studied in \cite{ref2} we easily get that $\pi_*(\mathbb{P}_\mu)$ is exact dimensional with
$$
\dim_e(\pi_*(\mathbb{P}_\mu)) = (q-1)^2 \sum_{p=1}^\infty \frac{H^\mu(\alpha^1_p)}{q^{p+1}}.
$$
Now we study $\mathbb{P}_\mu^y$. First observe that for $i$ such that $q \nmid i$, the map
$$
\phi_i : y \in \pi(X_\Omega) \longmapsto y|_{J_i} = (y_{q^\ell i})_{\ell=0}^\infty \in \pi(\Omega) 
$$
is measure-preserving, i.e. $(\phi_i)_*(\mathbb{P}_{\pi_* \mu}) = \pi_* \mu$. Let $p \geq n \geq 1$. For $(x,y) \in X_\Omega$ we have
\begin{align*}
\begin{split}
&\frac{\mathbb{P}_\mu \left(\left[(x_1,y_1) \cdots (x_n,y_n) y_{n+1} \cdots y_p \right]\right)}{\mathbb{P}_{\pi_* \mu}\left(\left[y_1 \cdots y_p\right]\right)} \\ & \\ &= \frac{\prod \limits_{\substack{i \leq p \\ q \nmid i}} \sum \limits_{x'_{q^k i}, \ldots, x'_{q^\ell i}} \mu \left( \left[(x_i,y_i) \cdots \left(x_{q^{k-1} i},y_{q^{k-1} i}\right) \left(x'_{q^{k} i},y_{q^{k} i}\right) \cdots \left(x'_{q^{\ell} i},y_{q^{\ell} i}\right) \right] \right)}{\prod \limits_{\substack{i \leq p \\ q \nmid i}} \sum \limits_{x'_i, \ldots, x'_{q^\ell i}} \mu \left( \left[(x'_i,y_i) \cdots \left(x'_{q^\ell i},y_{q^\ell i} \right) \right] \right)} \\ &= \prod\limits_{\substack{i \leq n \\ q \nmid i}} \frac{\mu \left(\left[(x_i,y_i) \cdots \left(x_{q^{k-1} i},y_{q^{k-1} i}\right) y_{q^k i} \cdots y_{q^\ell i} \right] \right)}{\pi_* \mu \left( \left[ y_i \cdots y_{q^\ell i} \right] \right)},
\end{split}
\end{align*}
where $q^{k-1} i \leq n < q^k i \leq q^{\ell} i \leq p < q^{\ell+1} i$. Using the remark above and letting $p \rightarrow \infty$ we deduce that for $\pi_*(\mathbb{P}_\mu) \text{-almost all} \ y$

$$
(\mathbb{P}_\mu)^y \left(\left[x|_n\right] \times \{y\}\right) = \prod\limits_{\substack{i \leq n \\ q \nmid i}} \mu^{y|_{J_i}}\left(\left[x|_{J_i^n} \right] \times \{y|_{J_i}\} \right).
$$
We will use the $\mathbb{P}_\mu$-almost everywhere defined i.i.d. random variables

$$
X_{i,n} : (x,y) \in X_\Omega \longmapsto -\log(\mu^{y|_{J_i}} \left( \left[x|_{J_i^n}\right] \times \{y|_{J_i}\}\right) \ \text{for} \ q \nmid i
$$
whose expectation is 
\begin{align*}
\begin{split}
&- \int_{\pi(X_\Omega)} \bigg( \int_{\pi^{-1}(\tilde{y})} \log \left(\mu^{\tilde{y}|_{J_i}} \left( \left[x|_{J_i^n}\right] \times \{\tilde{y}|_{J_i}\}\right) \right) d(\mathbb{P}_\mu^{\tilde{y}})(x,y) \bigg) d(\pi_*(\mathbb{P}_\mu))(\tilde{y}) \\ &=  \int_{\pi(X_\Omega)} H^{\mu^{\tilde{y}|_{J_i}}} \left(\Delta_p \left(\Omega_{\tilde{y}|_{J_i}}\right)\right) d(\pi_*(\mathbb{P}_\mu))(\tilde{y}) \\ &=  \int_{\pi(\Omega)} H^{\mu^y} \left(\Delta_p \left(\Omega_{y}\right)\right) d(\pi_*\mu)(y),
\end{split}
\end{align*}
where $\Omega_y = \pi^{-1}(\{y\}) \cap \Omega$ and $\Delta_p$ is the partition of $\Omega_y$ into cylinders of length $p$ on the first coordinate $x$, if $x|_{J_i^n}$ is of length $p$. Using again the same reasoning as in the one dimensional case when computing $\dim_H(\mathbb{P}_\mu)$ (see \cite{ref2}), we get that for $\pi_*(\mathbb{P}_\mu) \text{-almost all} \ y$, $\mathbb{P}_\mu^y$ is exact dimensional and
$$
\dim_e(\mathbb{P}_\mu^y) = (q-1)^2 \sum_{p=1}^\infty \int_{\pi(\Omega)} \frac{H^{\mu^y}\left(\Delta_p \left(\Omega_{y}\right)\right)}{q^{p+1}} d(\pi_*\mu)(y).
$$
Now we have
\begin{align*}
\begin{split}
&\int_{\pi(\Omega)} H^{\mu^y}\left(\Delta_p \left(\Omega_{y}\right)\right) d(\pi_*\mu)(y) \\&= - \int_{\pi(\Omega)} \sum_{I \in \theta_p(\Omega_y)} \mu^y(I) \log(\mu^y(I)) d(\pi_*\mu)(y) \\&= - \sum_{I \in \alpha^2_p} \int_{\pi(\Omega)} \mu^y(I \cap \pi^{-1}(\{y\})) \log(\mu^y(I \cap \pi^{-1}(\{y\}))) d(\pi_*\mu)(y) \\&= - \sum_{I \in \alpha^2_p} \int_{\pi(I)} \mu^y(I) \log(\mu^y(I)) d(\pi_*\mu)(y) \\ &\leq - \sum_{I \in \alpha^2_p} \pi_* \mu(\pi(I)) \bigg(\int_{\pi(I)} \frac{\mu^y(I)}{\pi_* \mu(\pi(I))} d(\pi_*\mu)(y) \bigg) \log \bigg(\int_{\pi(I)} \frac{\mu^y(I)}{\pi_* \mu(\pi(I))} d(\pi_*\mu)(y) \bigg) \\ &= - \sum_{I \in \alpha^2_p} \mu(I) \log \left(\frac{\mu(I)}{\pi_* \mu(\pi(I))} \right) \\&= H^\mu(\alpha^2_p|\alpha^1_p),
\end{split}
\end{align*}
using Jensen's inequality. The function $x \in \left[0,1\right] \mapsto -x\log(x)$ being strictly concave, this is a strict inequality unless for all $p \geq 1$, for all $I \in \alpha^2_p$, the map $y \in \pi(I) \mapsto \mu^y(I)$ is $\pi_* \mu$-almost surely constant.
\end{proof}

Using Lemma \ref{loc} we get

\begin{corollary}\label{LY}
If for all $p \geq 1$, for all $I \in \alpha^2_p$, the map $y \in \pi(I) \mapsto \mu^y(I)$ is almost surely constant, then $\mathbb{P}_\mu$ satisfies the Ledrappier-Young formula :
$$
\dim_H(\mathbb{P}_\mu) = \frac{1}{\log(m_1)} \dim_e(\mathbb{P}_\mu) + \left(\frac{1}{\log(m_2)}-\frac{1}{\log(m_1)}\right) \dim_e \left(\pi_* \left(\mathbb{P}_\mu\right)\right).
$$
\end{corollary}

This sufficient condition is equivalent to saying that for all $p \geq 1$, for all $I = \left[(x_1,y_1) \cdots (x_p,y_p)\right] \in \alpha^2_p$, for $\pi_* \mu$-almost all $y \in \pi(I)$ we have 
$$
\mu^y(I) = \frac{\mu(I)}{\pi_*\mu(\pi(I))} = \frac{\mu \left(\left[(x_1,y_1) \cdots (x_p,y_p)\right] \right)}{\mu \left(\left[y_1 \cdots y_p\right] \right)}.
$$
For instance, this is clearly satisfied when $\mu$ is an inhomogeneous Bernoulli product on $\Omega$. In this case $\mathbb{P}_\mu$ is not shift-invariant in general. However, we can easily build examples where the equality in Corollary \ref{LY} does not hold.

\begin{example}
Suppose that $j=0$. Then there exists $\Omega$ and $\mu$ a Borel probability measure on $\Omega$ such that
$$
\dim_H(\mathbb{P}_\mu) < \frac{1}{\log(m_1)} \dim_e(\mathbb{P}_\mu) + \left(\frac{1}{\log(m_2)}-\frac{1}{\log(m_1)}\right) \dim_e(\pi_*(\mathbb{P}_\mu)).
$$ 
\end{example}

Indeed, using the property $H^\mu(\alpha^2_{p-1} \lor \alpha^1_p) = H^\mu(\alpha^1_p|\alpha^2_{p-1}) + H^\mu(\alpha^2_{p-1})$ we have

$$
\dim_H(\mathbb{P}_\mu) = (q-1)^2 \sum_{p=1}^\infty \frac{H^\mu_{m_1}(\alpha^2_p)}{q^{p+1}} + (q-1)(1-\gamma)  \sum_{p=1}^\infty \frac{H^\mu_{m_2}(\alpha^1_p|\alpha^2_{p-1})}{q^{p}}
$$
and
\begin{align*}
\begin{split}
&\frac{1}{\log(m_1)} \dim_e(\mathbb{P}_\mu) + \left(\frac{1}{\log(m_2)}-\frac{1}{\log(m_1)}\right) \dim_e(\pi_*(\mathbb{P}_\mu)) \\&= (q-1)^2 \sum_{p=1}^\infty \frac{H^\mu_{m_1}(\alpha^2_p)}{q^{p+1}} + (q-1)^2 \left(\frac{1}{\log(m_2)}-\frac{1}{\log(m_1)}\right) \sum_{p=1}^\infty \frac{H^\mu(\alpha^1_p)}{q^{p+1}},
\end{split}
\end{align*}

It is then enough to choose $\Omega$ and $\mu$ such that 
\begin{itemize}
\item $H^\mu(\alpha^1_1) = 0$,
\item $H^\mu(\alpha^1_p | \alpha^2_{p-1}) = 0$ for all $p \geq 2$,
\item $H^\mu(\alpha^1_p) > 0$ for $p \geq 2$.
\end{itemize}
Such $\Omega$ and $\mu$ yield the desired example.

\subsection{Lower bound for $\dim_H(X_\Omega)$} \label{23}

We are now interested in maximizing $\dim_H(\mathbb{P}_\mu)$ over all Borel probability measures $\mu$ on $\Omega$. We define first the $j^\text{th}$ tree of prefixes of $\Omega$, which is a directed graph $\Gamma_j(\Omega)$ whose set of vertices is $\bigcup_{k=0}^\infty \text{Pref}_{k,j}(\Omega)$, where $\text{Pref}_{0,j}(\Omega) = \{\varnothing\}$. There is a directed edge from a prefix $$u = (x_1,y_1) \cdots (x_k,y_k) y_{k+1} \cdots y_{k+j}$$ to another one $v$ if $$v = (x_1,y_1) \cdots (x_k,y_k) (x_{k+1},y_{k+1}) y_{k+2} \cdots y_{k+j} y_{k+j+1}$$ for some $x_{k+1} \in \{0,\ldots,m_1-1\}$ and $y_{k+j+1} \in \{0,\ldots,m_2-1\}$. Moreover there is an edge from $\varnothing$ to every $u \in \text{Pref}_{1,j}(\Omega)$. $\Gamma_j(\Omega)$ is then a tree with its outdegree being bounded by $m_1 m_2$ (except the first edges from $\varnothing$, which can be more numerous). The following result is an analog of \cite[Lemma 2.1]{ref2}.

\begin{lemma} \label{arbre}
Let $u \in \rm{Pref}_{1,j}(\Omega)$ and $\Gamma_{u,j}(\Omega)$ be the tree of followers of $u$ in $\Gamma_j(\Omega)$. Let $V_{u,j}(\Omega)$ be its set of vertices. Then there exists a unique vector $t = t(u) \in \left[1,m_2^{\frac{2}{\gamma (q-1)}}\right]^{V_{u,j}(\Omega)}$ such that for all $(x_1,y_1) \cdots (x_k,y_k) y_{k+1} \cdots y_{k+j} \in V_{u,j}(\Omega)$
\begin{equation} \label{solution}
t_{(x_1,y_1) \cdots (x_k,y_k) y_{k+1} \cdots y_{k+j}}^{q^{j+1} \gamma} = \sum_{y'_{k+j+1}} \left( \sum_{x'_{k+1}} t_{(x_1,y_1) \cdots (x_k,y_k) (x'_{k+1},y_{k+1}) y_{k+2} \cdots y_{k+j} y'_{k+j+1}} \right)^{q^j \gamma},
\end{equation}
the sums being taken over the followers of $(x_1,y_1) \cdots (x_k,y_k) y_{k+1} \cdots y_{k+j}$ in $\Gamma_{u,j}(\Omega)$.

\end{lemma}

\begin{proof}
Let $Z = \left[1,m_2^{\frac{2}{\gamma (q-1)}}\right]^{V_{u,j}(\Omega)}$  and $F : Z \rightarrow Z$ be given by
$$
F(z_{(x_1,y_1) \cdots (x_k,y_k) y_{k+1} \cdots y_{k+j}}) = \left(\sum_{y'_{k+j+1}} \left( \sum_{x'_{k+1}} z_{(x_1,y_1) \cdots (x_k,y_k) (x'_{k+1},y_{k+1}) y_{k+2} \cdots y_{k+j} y'_{k+j+1}} \right)^{q^j \gamma} \right)^{\frac{1}{q^{j+1} \gamma}}.
$$
We can see that $F$ is monotone for the pointwise partial order $\leq$, defined as $$z \leq z' \Leftrightarrow \forall v \in V_{u,j}(\Omega), \ z_v \leq z'_v$$ for $z,z' \in Z$. Indeed since $q^j \gamma$, $\frac{1}{q^{j+1} \gamma} \geq 0$ we have
$$
z \leq z' \Longrightarrow F(z) \leq F(z').
$$
Denote by $1$ the constant function equal to $1$ over $Z$. Then $1 \leq F(1) \leq F^2(1) \leq \cdots$, so by compactness $(F^n(1))_{n \geq 1}$ has a pointwise limit $t$, which is a fixed point of $F$. Let us now verify the uniqueness. Suppose that $t$ and $t'$ are two fixed points of $F$ and that $t$ is not smaller than $t'$ for $\leq$ (without loss of generality). Let 
$$
\omega = \inf \{ \xi > 1, \ t \leq \xi t' \}.
$$
Clearly $\omega \leq m_2^{\frac{2}{\gamma (q-1)}}$, and by continuity we have $t \leq \omega t'$, so $\omega > 1$. Now
$$
t = F(t) \leq F(\omega t') = \omega^{\frac{1}{q}} F(t') = \omega^{\frac{1}{q}} t',
$$
contradicting the definition of $\omega$. \newline
\end{proof}

Furthermore we define 
$$
t_\varnothing = \sum_{y'_1} \bigg( \sum_{y'_2} \bigg( \cdots \bigg( \sum_{y'_{j+1}} \bigg( \sum_{x'_1} t_{(x'_1,y'_1) y'_2 \cdots y'_{j+1}} \bigg)^{q^j \gamma} \bigg)^{\frac{1}{q}} \cdots \bigg)^\frac{1}{q} \bigg)^\frac{1}{q}.\\
$$

\begin{proposition} \label{mesure}
For $u = (x_1,y_1) \cdots (x_k,y_k) y_{k+1} \cdots y_{k+j} \in \textup{Pref}_{k,j}(\Omega)$ define
\begin{align*}
\begin{split}
\mu(\left[u\right]) = &\frac{t_{(x_1,y_1) y_2 \cdots y_{j+1}} \left(\sum_{x'_1} t_{(x'_1,y_1) y_2 \cdots y_{j+1}} \right)^{q^j \gamma-1}}{t_\varnothing} \\ &\boldsymbol{\cdot}\prod_{p=0}^{j-1} \bigg(\sum_{y'_{j+1-p}} \bigg( \sum_{y'_{j+2-p}} \bigg( \cdots \bigg(\sum_{y'_{j+1}} \bigg(\sum_{x'_1} t_{(x'_1,y_1) y_2 \cdots y_{j-p} y'_{j+1-p} \cdots y'_{j+1}} \bigg)^{q^j \gamma} \bigg)^{\frac{1}{q}} \cdots \bigg)^{\frac{1}{q}} \bigg)^{\frac{1}{q}} \bigg)^{\frac{1-q}{q}} \\ &\boldsymbol{\cdot} \prod_{p=2}^k \frac{t_{(x_1,y_1) \cdots (x_p,y_p) y_{p+1} \cdots y_{p+j}} \left(\sum_{x'_p} t_{(x_1,y_1) \cdots (x'_p,y_p) y_{p+1} \cdots y_{p+j}} \right)^{q^j \gamma-1}}{t_{(x_1,y_1) \cdots (x_{p-1},y_{p-1}) y_{p} \cdots y_{p-1+j}}^{q^{j+1} \gamma}},
\end{split}
\end{align*}
where there are $p+2$ sums and $p$ exponents $\frac{1}{q}$ in each term of the first product. This defines a Borel probability measure on $\Omega$ such that $\mathbb{P}_\mu$ is the unique optimal measure, i.e. such that $\dim_H(\mathbb{P}_\mu)$ is maximal over all Borel probability measures $\mu$ on $\Omega$. Moreover we have $\dim_H(\mathbb{P}_\mu) = \frac{q-1}{q} \log_{m_2}(t_\varnothing)$. Using Theorem \ref{calme} we deduce that 
$$
\dim_H(X_\Omega) \geq \frac{q-1}{q} \log_{m_2}(t_\varnothing).
$$
\end{proposition}

\begin{proof}
Let
\begin{align*}
\begin{split}
S(\Omega,\mu) = (q-1)^2 \sum_{p=1}^j \frac{H^\mu_{m_2}(\alpha^1_p)}{q^{p+1}} &+ (q-1)(1-q^j \gamma) \sum_{p=j+1}^\infty \frac{H^\mu_{m_2}(\alpha^2_{p-j-1} \lor \alpha^1_p)}{q^p} \\ &+ (q-1)(q^{j+1} \gamma-1) \sum_{p=j+1}^\infty \frac{H^\mu_{m_2}(\alpha^2_{p-j} \lor \alpha^1_p)}{q^{p+1}}.
\end{split}
\end{align*}
We try to optimize $S(\Omega,\mu)$ over all Borel probability measures $\mu$ on $\Omega$. Let $S(\Omega) = \max_\mu S(\Omega,\mu)$. Recall that for some measurable partitions $\mathcal{P}$,$\mathcal{Q}$ of $\Omega$ we have
$$
H^\mu_{m_2}(\mathcal{P}|\mathcal{Q}) = \sum_{Q \in \mathcal{Q}} \left( - \sum_{P \in \mathcal{P}} \mu(P|Q) \log_{m_2}(\mu(P|Q)) \right) \mu(Q).
$$ 
Let $p \geq j+2$. We have 
$$
H^\mu_{m_2}(\alpha^2_{p-j-1} \lor \alpha^1_p) = H^\mu_{m_2}(\alpha^2_{p-j-1} \lor \alpha^1_p|\alpha^2_1 \lor \alpha^1_{j+1}) + H^\mu_{m_2}(\alpha^2_1 \lor \alpha^1_{j+1})
$$
and 
$$
H^\mu_{m_2}(\alpha^2_{p-j} \lor \alpha^1_p) = H^\mu_{m_2}(\alpha^2_{p-j} \lor \alpha^1_p|\alpha^2_1 \lor \alpha^1_{j+1}) + H^\mu_{m_2}(\alpha^2_1 \lor \alpha^1_{j+1}).
$$
Moreover

\begin{align*}
\begin{split}
&H^\mu_{m_2}(\alpha^2_{p-j-1} \lor \alpha^1_p|\alpha^2_1 \lor \alpha^1_{j+1}) \\&= \sum_{x_1,y_1,y_2,\ldots,y_{j+1}} \theta_{(x_1,y_1) y_2 \cdots y_{j+1}} H^{\mu_{(x_1,y_1)y_2 \cdots y_{j+1}}}_{m_2}\left(\alpha^2_{p-j-2} \lor \alpha^1_{p-1} \left(\Omega_{(x_1,y_1)y_2 \cdots y_{j+1}}\right)\right)
\end{split}
\end{align*}
and

\begin{align*}
\begin{split}
&H^\mu_{m_2}(\alpha^2_{p-j} \lor \alpha^1_p|\alpha^2_1 \lor \alpha^1_{j+1}) \\&= \sum_{x_1,y_1,y_2,\ldots,y_{j+1}} \theta_{(x_1,y_1) y_2 \cdots y_{j+1}} H^{\mu_{(x_1,y_1) y_2 \cdots y_{j+1}}}_{m_2}\left(\alpha^2_{p-j-1} \lor \alpha^1_{p-1}\left(\Omega_{(x_1,y_1)y_2 \cdots y_{j+1}}\right)\right),
\end{split}
\end{align*}
where $$\theta_{(x_1,y_1)y_2 \cdots y_{j+1}} = \mu(\left[(x_1,y_1)y_2 \cdots y_{j+1}\right]),$$ and $H_{m_2}^{\mu_{(x_1,y_1) y_2,\cdots y_{j+1}}}\left(\alpha^2_{p-j-2} \lor \alpha^1_{p-1}\left(\Omega_{(x_1,y_1) y_2 \cdots y_{j+1}}\right)\right)$ is the entropy of the partition of $\Omega_{(x_1,y_1) y_2 \cdots y_{j+1}}$, the follower set of $(x_1,y_1)$ in $\Omega$ with $y_2 \cdots y_{j+1}$ being fixed, with respect to $\mu_{(x_1,y_1)y_2 \cdots y_{j+1}}$ which is the normalized measure induced by $\mu$ on $\Omega_{(x_1,y_1) y_2 \cdots y_{j+1}}$. Then
\begin{align*}
\begin{split}
S(\Omega,\mu) = (q-1)^2 \sum_{p=1}^j \frac{H^\mu_{m_2}(\alpha^1_p)}{q^{p+1}} &+ \frac{(q-1)(1-q^j \gamma)}{q^{j+1}} H^\mu_{m_2}(\alpha^1_{j+1}) + \frac{\gamma (q-1)}{q} H^\mu_{m_2}(\alpha^2_1 \lor \alpha^1_{j+1}) \\ &+ \frac{1}{q} \sum_{x_1,y_1,y_2,\ldots,y_{j+1}} \theta_{(x_1,y_1) y_2 \cdots y_{j+1}} S \left(\Omega_{(x_1,y_1) y_2 \cdots y_{j+1}},\mu_{(x_1,y_1) y_2 \cdots y_{j+1}}\right).
\end{split}
\end{align*}
Observe that the measure is completely determined by the knowledge of $\theta_{(x_1,y_1) y_2 \cdots y_{j+1}}$ and $\mu_{(x_1,y_1) y_2 \cdots y_{j+1}}$ for all $(x_1,y_1) y_2 \cdots y_{j+1}$. The optimization problems on $\Omega_{(x_1,y_1) y_2 \cdots y_{j+1}}$ being independent we get
\begin{align*}
\begin{split}
S(\Omega) = \max_{\theta_{(x_1,y_1) y_2 \cdots y_{j+1}}} (q-1)^2 \sum_{p=1}^j \frac{H^\mu_{m_2}(\alpha^1_p)}{q^{p+1}} &+ \frac{(q-1)(1-q^j \gamma)}{q^{j+1}} H^\mu_{m_2}(\alpha^1_{j+1}) \\&+ \frac{\gamma (q-1)}{q} H^\mu_{m_2}(\alpha^2_1 \lor \alpha^1_{j+1}) \\ &+ \frac{1}{q} \sum_{x_1,y_1,y_2,\ldots,y_{j+1}} \theta_{(x_1,y_1) y_2 \cdots y_{j+1}} S \left(\Omega_{(x_1,y_1) y_2 \cdots y_{j+1}}\right).
\end{split}
\end{align*}
After factorizing, we have
\begin{align*}
\begin{split}
S(\Omega) &= \max \frac{q-1}{q} \bigg( H^\mu_{m_2}(\beta_1) + \frac{1}{q} \sum_{y_1} \theta_{y_1} \bigg( -\sum_{y_2} \frac{\theta_{y_1 y_2}}{\theta_{y_1}} \log_{m_2} \left(\frac{\theta_{y_1 y_2}}{\theta_{y_1}} \right) \\&+ \frac{1}{q} \sum_{y_2} \frac{\theta_{y_1 y_2}}{\theta_{y_1}} \bigg(-\sum_{y_3} \frac{\theta_{y_1 y_2 y_3}}{\theta_{q_1 y_2}} \log_{m_2} \left(\frac{\theta_{y_1 y_2 y_3}}{\theta_{q_1 y_2}} \right) \\ &+ \frac{1}{q} \sum_{y_3} \frac{\theta_{y_1 y_2 y_3}}{\theta_{y_1 y_2}} \bigg( \cdots + q^j \gamma \sum_{y_{j+1}} \frac{\theta_{y_1 \cdots y_{j+1}}}{\theta_{y_1 \cdots y_j}} \bigg( -\sum_{x_1} \frac{\theta_{(x_1,y_1) y_2 \cdots y_{j+1}}}{\theta_{y_1 \cdots y_{j+1}}} \log_{m_2} \left(\frac{\theta_{(x_1,y_1) y_2 \cdots y_{j+1}}}{\theta_{y_1 \cdots y_{j+1}}} \right) \\ &+ \frac{1}{\gamma (q-1)} \sum_{x_1} \frac{\theta_{(x_1,y_1) y_2 \cdots y_{j+1}}}{\theta_{y_1 \cdots y_{j+1}}} S \left( \Omega_{(x_1,y_1) y_2 \cdots y_{j+1}} \right) \bigg) \cdots \bigg) \bigg) \bigg) \bigg).
\end{split}
\end{align*}
We can now recursively optimize these quantities. First fix $y_1, \ldots, y_{j+1}$. To optimize the last part of the above expression of $S(\Omega)$, we use Lemma \ref{convexe} and we obtain
$$
\frac{\theta_{(x_1,y_1) y_2 \cdots y_{j+1}}}{\theta_{y_1 \cdots y_{j+1}}} = \frac{m_2^{\frac{S \left( \Omega_{(x_1,y_1) y_2 \cdots y_{j+1}}\right)}{\gamma (q-1)}}}{\sum_{x'_1} m_2^{\frac{S \left( \Omega_{(x'_1,y_1) y_2 \cdots y_{j+1}} \right)}{\gamma(q-1)}}}
$$
and

\begin{align*}
\begin{split}
-\sum_{x_1} \frac{\theta_{(x_1,y_1) y_2 \cdots y_{j+1}}}{\theta_{y_1 \cdots y_{j+1}}} \log_{m_2} \left(\frac{\theta_{(x_1,y_1) y_2 \cdots y_{j+1}}}{\theta_{y_1 \cdots y_{j+1}}} \right) &+ \frac{1}{\gamma (q-1)} \sum_{x_1} \frac{\theta_{(x_1,y_1) y_2 \cdots y_{j+1}}}{\theta_{y_1 \cdots y_{j+1}}} S \left( \Omega_{(x_1,y_1) y_2 \cdots y_{j+1}} \right) \\ &= \log_{m_2} \left(\sum_{x_1} m_2^{\frac{S \left( \Omega_{(x_1,y_1) y_2 \cdots y_{j+1}} \right)}{\gamma(q-1)}} \right).
\end{split}
\end{align*}
Using again Lemma \ref{convexe}, we get $\frac{\theta_{y_1 \cdots y_{j+1}}}{\theta_{y_1 \cdots y_j}}$, and so on. This gives us the weights $\theta_{(x_1,y_1) y_2 \cdots y_{j+1}}$, which are equal to 
\begin{align*}
\begin{split}
&\frac{z_{(x_1,y_1) y_2 \cdots y_{j+1}} \left(\sum_{x'_1} z_{(x'_1,y_1) y_2 \cdots y_{j+1}} \right)^{q^j \gamma-1}}{z_\varnothing} \\&\boldsymbol{\cdot} \prod_{p=0}^{j-1} \bigg(\sum_{y'_{j+1-p}} \bigg( \sum_{y'_{j+2-p}} \bigg( \cdots \bigg(\sum_{y'_{j+1}} \bigg(\sum_{x'_1} z_{(x'_1,y_1) y_2 \cdots y_{j-p} y'_{j+1-p} \cdots y'_{j+1}} \bigg)^{q^j \gamma} \bigg)^{\frac{1}{q}} \cdots \bigg)^{\frac{1}{q}} \bigg)^{\frac{1}{q}} \bigg)^{\frac{1-q}{q}},
\end{split}
\end{align*}
where $z_{(x_1,y_1) y_2 \cdots y_{j+1}} = m_2^{\frac{S \left(\Omega_{(x_1,y_1) y_2 \cdots y_{j+1}}\right)}{\gamma (q-1)}}$ and $z_\varnothing = m_2^{\frac{q S(\Omega)}{q-1}}$. In particular we get
$$
z_\varnothing = \sum_{y'_1} \bigg( \sum_{y'_2} \bigg( \cdots \bigg( \sum_{y'_{j+1}} \bigg( \sum_{x'_1} z_{(x'_1,y'_1) y'_2 \cdots y'_{j+1}} \bigg)^{q^j \gamma} \bigg)^{\frac{1}{q}} \cdots \bigg)^\frac{1}{q} \bigg)^\frac{1}{q}.
$$
Now let us consider $\Omega_u$ for fixed $u = (x_1,y_1) y_2 \cdots y_{j+1} \in \text{Pref}_{1,j}(\Omega)$. The optimization problem is now analogous on this tree, but simpler : we now have to optimize the quantity
\begin{align*}
\begin{split}
\frac{(q-1)(1-q^j \gamma)}{q^{j+1}} &H^{\mu_{(x_1,y_1)y_2 \cdots y_{j+1}}}_{m_2} \left(\alpha^1_{j+1} \left(\Omega_{(x_1,y_1)y_2 \cdots y_{j+1}} \right) \right) \\&+ \frac{\gamma (q-1)}{q} H^{\mu_{(x_1,y_1)y_2 \cdots y_{j+1}}}_{m_2} \left(\alpha^2_1 \lor \alpha^1_{j+1} \left(\Omega_{(x_1,y_1)y_2 \cdots y_{j+1}} \right) \right) \\&+ \frac{1}{q} \sum_{x_2,y_{j+2}} \frac{\theta_{(x_1,y_1)(x_2,y_2) y_3 \cdots y_{j+2}}}{\theta_{(x_1,y_1)y_2 \cdots y_{j+1}}} S \left(\Omega_{(x_1,y_1)(x_2,y_2) y_3 \cdots y_{j+2}} \right), 
\end{split}
\end{align*}
which is after factorization

\begin{align*}
\begin{split}
 \frac{q-1}{q^{j+1}} \bigg(  &-\sum_{y_{j+2}} \frac{\theta_{(x_1,y_1) y_2 \cdots y_{j+2}}}{\theta_{(x_1,y_1) y_2 \cdots y_{j+1}}} \log_{m_2}\left(\frac{\theta_{(x_1,y_1) y_2 \cdots y_{j+2}}}{\theta_{(x_1,y_1) y_2 \cdots y_{j+1}}}\right) \\ &+ q^j \gamma \sum_{y_{j+2}} \frac{\theta_{(x_1,y_1) y_2 \cdots y_{j+2}}}{\theta_{(x_1,y_1) y_2 \cdots y_{j+1}}} \bigg(-\sum_{x_2} \frac{\theta_{(x_1,y_1)(x_2,y_2) y_3 \cdots y_{j+2}}}{\theta_{(x_1,y_1) y_2 \cdots y_{j+2}}} \log_{m_2} \bigg(\frac{\theta_{(x_1,y_1)(x_2,y_2) y_3 \cdots y_{j+2}}}{\theta_{(x_1,y_1) y_2 \cdots y_{j+2}}} \bigg) \\ &+ \frac{1}{\gamma (q-1)} \sum_{x_2} \frac{\theta_{(x_1,y_1)(x_2,y_2) y_3 \cdots y_{j+2}}}{\theta_{(x_1,y_1) y_2 \cdots y_{j+2}}} S \left(\Omega_{(x_1,y_1)(x_2,y_2) y_3 \cdots y_{j+2}} \right) \bigg) \bigg).
\end{split}
\end{align*}
This gives the weights $$\frac{\theta_{(x_1,y_1)(x_2,y_2) y_3 \cdots y_{j+2}}}{\theta_{(x_1,y_1) y_2 \cdots y_{j+1}}} = \frac{z_{(x_1,y_1)(x_2,y_2) y_3 \cdots y_{j+2}} \left(\sum_{x'_2} z_{(x_1,y_1)(x'_2,y_2) y_3 \cdots y_{j+2}} \right)^{q^j \gamma-1}}{z_{(x_1,y_1) y_2 \cdots y_{j+1}}^{q^{j+1} \gamma}},$$ with $z_{(x_1,y_1)(x_2,y_2) y_3 \cdots y_{j+2}} = m_2^{\frac{S \left(\Omega_{(x_1,y_1)(x_2,y_2) y_3 \cdots y_{j+2}}\right)}{\gamma (q-1)}}$, $z_{(x_1,y_1) y_2 \cdots y_{j+1}} = m_2^{\frac{S \left(\Omega_{(x_1,y_1) y_2 \cdots y_{j+1}}\right)}{\gamma (q-1)}}$ and
$$
z_{(x_1,y_1) y_2 \cdots y_{j+1}}^{q^{j+1} \gamma} = \sum_{y'_{j+2}} \left( \sum_{x'_{2}} z_{(x_1,y_1) (x'_2,y_2) y_3 \cdots y'_{j+2}} \right)^{q^j \gamma}.
$$
This is exactly equation \eqref{solution} at the root of the graph $\Gamma_{u,j}(\Omega)$. The problem being the same at each vertex for $\Gamma_{u,j}(\Omega)$, for all $u \in \text{Pref}_{1,j}(\Omega)$, we can repeat the argument for the entire graphs. We also get the given formula for the optimal measure from the form of all optimal probability vectors that we found. The solutions $z = z(u)$ of the systems \eqref{solution} which we get this way are in $\left[1,m_2^{\frac{2}{\gamma(q-1)}} \right]^{V_{u,j}(\Omega)}$, thus we have $z(u) = t(u)$ for all $u$ (indeed for all $k \geq 1$, for all $v \in \text{Pref}_{k,j}(\Omega)$, for all $\mu$ on $\Omega_v$ we have $\dim_H(\mathbb{P}_\mu) \leq 2$, so $S(\Omega_v) \leq 2$).\newline

\end{proof}

\subsection{Upper bound for $\dim_H(X_\Omega)$} \label{24}

\begin{theorem}
Let $\mu$ be the Borel probability measure on $\Omega$ defined in the last theorem, and let $\mathbb{P}_\mu$ be the corresponding Borel probability measure on $X_\Omega$. Let $(x,y) \in X_\Omega$. Then 
$$
\liminf_{n \rightarrow \infty} \frac{-\log_{m_2}(\mathbb{P}_\mu(B_n(x,y)))}{L(n)} \leq \frac{q-1}{q} \log_{m_2}(t_\varnothing),
$$
from which we deduce that $\dim_H(X_\Omega)=\frac{q-1}{q} \log_{m_2}(t_\varnothing)$.
\end{theorem}

\begin{proof}

Recall that

$$
-\log_{m_2}(\mathbb{P}_\mu(B_n(x,y))) = - \sum_{\substack{i,\ q \nmid i \\ i \leq L(n)}} \log_{m_2} \left( \mu \left(\left[(x_i,y_i) (x_{q i},y_{q i}) \cdots (x_{q^{k-1} i},y_{q^{k-1} i}) y_{q^{k} i} \cdots y_{q^\ell i} \right] \right) \right),
$$
where $k$ and $\ell$ are determined by $i < qi < \cdots < q^{k-1} i \leq n < q^{k} i < \dots < q^\ell i \leq L(n) < q^{\ell+1} i$ in each term of the sum.

Suppose first that $j=1$ for the sake of simplicity. We have
\begin{align*}
\begin{split}
\mu \left(\left[(x_1,y_1) \cdots (x_k,y_k) y_{k+1}\right] \right) &= \frac{t_{(x_1,y_1) y_2} \left(\sum_{x'_1} t_{(x'_1,y_1) y_2} \right)^{q \gamma-1} \left(\sum_{y'_2} \left( \sum_{x'_1} t_{(x'_1,y_1) y'_2} \right)^{q \gamma} \right)^{\frac{1-q}{q}}}{t_\varnothing} \\ &\boldsymbol{\cdot} \prod_{p=2}^k \frac{t_{(x_1,y_1) \cdots (x_p,y_p) y_{p+1}} \left(\sum_{x'_p} t_{(x_1,y_1) \cdots (x'_p,y_p) y_{p+1}} \right)^{q \gamma-1}}{t_{(x_1,y_1) \cdots (x_{p-1},y_{p-1}) y_p}^{q^2 \gamma}},
\end{split}
\end{align*}

\begin{align*}
\begin{split}
\mu \left(\left[(x_1,y_1) \cdots (x_{k-1},y_{k-1}) y_k y_{k+1}\right] \right) &= \frac{t_{(x_1,y_1) y_2} \left(\sum_{x'_1} t_{(x'_1,y_1) y_2} \right)^{q \gamma-1} \left(\sum_{y'_2} \left( \sum_{x'_1} t_{(x'_1,y_1) y'_2} \right)^{q \gamma} \right)^{\frac{1-q}{q}}}{t_\varnothing} \\ &\boldsymbol{\cdot} \prod_{p=2}^{k-1} \frac{t_{(x_1,y_1) \cdots (x_p,y_p) y_{p+1}} \left(\sum_{x'_p} t_{(x_1,y_1) \cdots (x'_p,y_p) y_{p+1}} \right)^{q \gamma-1}}{t_{(x_1,y_1) \cdots (x_{p-1},y_{p-1}) y_p}^{q^2 \gamma}} \\ &\boldsymbol{\cdot} \frac{\left(\sum_{x'_k} t_{(x_1,y_1) \cdots (x'_k,y_k) y_{k+1}} \right)^{q \gamma}}{t^{q^2 \gamma}_{(x_1,y_1) \cdots (x_{k-1},y_{k-1}) y_k}}
\end{split}
\end{align*}
for $k \geq 2$,

$$
\mu \left(\left[y_1 y_2 \right] \right) = \frac{ \left(\sum_{x'_1} t_{(x'_1,y_1) y_2} \right)^{q \gamma} \left(\sum_{y'_2} \left( \sum_{x'_1} t_{(x'_1,y_1) y'_2} \right)^{q \gamma} \right)^{\frac{1-q}{q}}}{t_\varnothing},
$$
and

$$
\mu \left(\left[y_1 \right] \right) = \frac{ \left(\sum_{y'_2} \left( \sum_{x'_1} t_{(x'_1,y_1) y'_2} \right)^{q \gamma} \right)^{\frac{1}{q}}}{t_\varnothing}.
$$
For each positive integer $\kappa \leq L(n)$, we can write $\kappa = q^r i$ with $q \nmid i$ for some unique $(r,i)$. Now, developing the product $\mathbb{P}_\mu(B_n(x,y))$, we pick up
\begin{itemize}
\item $\frac{1}{t_\varnothing}$ for each $i \leq L(n)$ such that $q \nmid i$,
\item $t_{(x_i,y_i) \cdots (x_{q^r i},y_{q^r i}) y_{q^{r+1} i}}$ for each $\kappa = q^r i \leq n$,
\item $\frac{1}{t_{(x_i,y_i) \cdots \left(x_{q^r i},y_{q^r i}\right) y_{q^{r+1} i}}^{q^2 \gamma}}$ for each $\kappa \leq \left\lfloor \frac{L(n)}{q^2} \right\rfloor$ : that is because for these $\kappa$ we have $q^2 \kappa = q^{r+2} i \leq L(n)$, and for $\kappa > \left\lfloor\frac{L(n)}{q^2} \right\rfloor$ we have $q^2 \kappa \geq q^2 \left\lfloor \frac{L(n)}{q^2} \right\rfloor + q^2 > L(n)$,
\item $\left(\sum_{x'_{q^r i}} t_{(x_i,y_i) \cdots (x'_{q^r i},y_{q^r i}) y_{q^{r+1} i}} \right)^{q\gamma - 1}$ for each $\kappa \leq n$,
\item $\left(\sum_{x'_{q^r i}} t_{(x_i,y_i) \cdots (x'_{q^r i},y_{q^r i}) y_{q^{r+1} i}} \right)^{q\gamma}$ for each $n < \kappa \leq \left\lfloor \frac{L(n)}{q} \right\rfloor$,
\item $\left(\sum_{y'_{qi}} \left( \sum_{x'_i} t_{(x'_i,y_i) y'_{qi}} \right)^{q \gamma} \right)^{\frac{1-q}{q}}$ for each $i \leq \left\lfloor \frac{L(n)}{q} \right\rfloor$ such that $q \nmid i$,
\item $\left(\sum_{y'_{qi}} \left( \sum_{x'_i} t_{(x'_i,y_i) y'_{qi}} \right)^{q \gamma} \right)^{\frac{1}{q}}$ for each $\left\lfloor \frac{L(n)}{q} \right\rfloor < i \leq L(n)$ such that $q \nmid i$.
\end{itemize}
Thus if we define 
 
$$R(\kappa) = \log_{m_2} \left(t_{(x_i,y_i) (x_{qi},y_{qi}) \cdots (x_{q^r i},y_{q^r i}) y_{q^{r+1} i}} \right)$$ for $\kappa = q^r i$ with $q \nmid i$, $$ \tilde{R}(\kappa) = \log_{m_2} \left(\sum_{x'_{q^r i}} t_{(x_i,y_i) (x_{qi},y_{qi}) \cdots (x'_{q^r i},y_{q^r i}) y_{q^{r+1} i}} \right)$$ for $\kappa = q^r i$ with $q \nmid i$, and
$$
u^1_n = \frac{1}{n} \sum_{\kappa=1}^n R(\kappa), \hspace{1cm} u^2_n = \frac{1}{n} \sum_{\kappa=1}^n \tilde{R}(\kappa),
$$

$$
u^3_n = \frac{1}{n} \sum_{i \leq n, \ q \nmid i} \log_{m_2} \bigg(\sum_{y'_{qi}} \bigg(\sum_{x'_i} t_{(x'_i,y_i) y'_{qi}} \bigg)^{q \gamma} \bigg),
$$
we get
\begin{align*}
\begin{split}
\log_{m_2} \left(\mathbb{P}_\mu(B_n(x,y)) \right) &= n u^1_n - \gamma q^2 \left\lfloor \frac{L(n)}{q^2} \right\rfloor u^1_{\left\lfloor \frac{L(n)}{q^2} \right\rfloor} + \gamma q \left \lfloor \frac{L(n)}{q} \right\rfloor u^2_{\left\lfloor \frac{L(n)}{q} \right\rfloor} - n u^2_n \\ &+ \frac{1}{q} L(n) u^3_{L(n)} - \left\lfloor \frac{L(n)}{q} \right\rfloor u^3_{\left\lfloor \frac{L(n)}{q} \right\rfloor} - \# \{i \in \llbracket 1, L(n) \rrbracket, \ q \nmid i \} \log_{m_2}(t_\varnothing).
\end{split}
\end{align*}
Getting back to the general case, let us define $j+2$ sequences as follows. At first, set

$$
u^1_n = \frac{1}{n} \sum_{\kappa=1}^n R(\kappa) \hspace{1cm} u^2_n = \frac{1}{n} \sum_{\kappa=1}^n \tilde{R}(\kappa),
$$
where $$R(\kappa) = \log_{m_2} \left(t_{(x_i,y_i) (x_{qi},y_{qi}) \cdots (x_{q^r i},y_{q^r i}) y_{q^{r+1} i} \cdots y_{q^{r+j} i}} \right)$$ if $\kappa = q^r i$ with $q \nmid i$, and $$ \tilde{R}(\kappa) = \log_{m_2} \left(\sum_{x'_{q^r i}} t_{(x_i,y_i) (x_{qi},y_{qi}) \cdots (x'_{q^r i},y_{q^r i}) y_{q^{r+1} i} \cdots y_{q^{r+j} i}} \right)$$ if $\kappa = q^r i$ with $q \nmid i$. Then, for $3 \leq k \leq j+2$ let
$$
u^k_n = \frac{1}{n} \sum_{i \leq n, \ q \nmid i} \log_{m_2} \bigg(\sum_{y'_{q^{j+3-k} i}} \bigg( \sum_{y'_{q^{j+4-k} i}} \bigg( \cdots \bigg(\sum_{y'_{q^j i}} \bigg(\sum_{x'_i} t_{(x'_i,y_i) y_{qi} \cdots y'_{q^{j} i}} \bigg)^{q^j \gamma} \bigg)^{\frac{1}{q}} \cdots \bigg)^{\frac{1}{q}} \bigg)^{\frac{1}{q}} \bigg),
$$
where there are exactly $k-1$ sums and $k-3$ exponents $\frac{1}{q}$ in each $\log_{m_2}$ term. It is easy to see that all these sequences are nonnegative, bounded, with
$$
\forall 1 \leq k \leq j+2, \ \lim_{n \rightarrow \infty} u^k_{n+1} - u^k_n = 0.
$$
Let $\epsilon > 0$. Using the definition of $\mu$ we can get the following expression for $n$ large enough, which will be justified when studying the case $d \geq 2$
\begin{align*}
\begin{split}
\frac{-\log_{m_2}(\mathbb{P}_\mu(B_n(x,y)))}{L(n)} &= \gamma \frac{q^{j+1}}{L(n)} \left\lfloor \frac{L(n)}{q^{j+1}} \right\rfloor u^1_{\left\lfloor \frac{L(n)}{q^{j+1}} \right\rfloor} - \frac{n}{L(n)} u^1_n + \frac{n}{L(n)} u^2_n - \gamma \frac{q^j}{L(n)} \left\lfloor \frac{L(n)}{q^j} \right\rfloor u^2_{\left\lfloor \frac{L(n)}{q^j} \right\rfloor}\\ &+ \frac{1}{L(n)} \sum_{k=0}^{j-1} \left( \left\lfloor \frac{L(n)}{q^{j-k}} \right\rfloor u^{k+3}_{\left\lfloor \frac{L(n)}{q^{j-k}} \right\rfloor} - \frac{1}{q} \left\lfloor \frac{L(n)}{q^{j-k-1}} \right\rfloor u^{k+3}_{\left\lfloor \frac{L(n)}{q^{j-k-1}} \right\rfloor} \right)  \\ &+ \frac{ \# \{i \in \llbracket 1, L(n) \rrbracket, \ q \nmid i \}}{L(n)} \log_{m_2}(t_\varnothing) \\ &\leq \gamma \left(u^1_{\left\lfloor \frac{L(n)}{q^{j+1}} \right\rfloor} - u^1_n \right) + \gamma \left(u^2_n - u^2_{\left\lfloor \frac{L(n)}{q^{j}} \right\rfloor} \right) \\&+ \sum_{k=0}^{j-1} \frac{1}{q^{j-k}} \left(u^{k+3}_{\left \lfloor \frac{L(n)}{q^{j-k}} \right \rfloor} - u^{k+3}_{\left \lfloor \frac{L(n)}{q^{j-k-1}} \right \rfloor} \right) \\ &+ \frac{q-1}{q} \log_{m_2}(t_\varnothing) + \epsilon.
\end{split}
\end{align*}
To conclude we now use Lemma \ref{combi} and then let $\epsilon \rightarrow 0$.

\end{proof}

\begin{example}
If $\Omega$ is a Sierpi\'nski carpet, then clearly $X_\Omega = \Omega$. Using uniqueness in Theorem \ref{arbre} we deduce that the values $t_{(x_1,y_1) y_2 \dots y_{j+1}}$ do not depend on $x_1$ and $y_1$. We call them $t_{y_2 \cdots y_{j+1}}$. Equation \eqref{solution} now reduces to
$$
t_{y_2 \cdots y_{j+1}}^{q^{j+1} \gamma} = N(y_2)^{q^j \gamma} \sum_{y_{j+2}} t_{y_3 \cdots y_{j+2}}^{q^j \gamma},
$$
where $N(y_2) = \# \{x_2, \ (x_2,y_2) \in A\}$. Thus
$$
\sum_{y_{j+1}} t_{y_2 \cdots y_{j+1}}^{q^j \gamma} = N(y_2)^{q^{j-1} \gamma} \sum_{y_{j+1}} \bigg( \sum_{y_{j+2}} t_{y_3 \cdots y_{j+2}}^{q^j \gamma} \bigg)^\frac{1}{q},
$$
and so on. After having summed on the different coordinates we get
$$
\sum_{y_2} \bigg( \sum_{y_3} \bigg( \cdots \bigg( \sum_{y_j} \bigg( \sum_{y_{j+1}} t_{y_2 \cdots y_{j+1}}^{q^j \gamma} \bigg)^{\frac{1}{q}} \bigg)^{\frac{1}{q}} \cdots \bigg)^{\frac{1}{q}} \bigg)^{\frac{1}{q}} = \left( \sum_{y_2} N(y_2)^\gamma \right)^{\frac{q}{q-1}}.
$$
So finally $t_\varnothing = \left( \sum_{y_2} N(y_2)^\gamma \right)^{\frac{q}{q-1}}$ and $\dim_H(X_\Omega) = \log_{m_2} \left( \sum_{y_2} N(y_2)^\gamma \right)$, which is as expected in the McMullen formula. Also, we check that the maximizing measure is the Bernoulli product measure used by McMullen. \newline
\end{example} 

\begin{example}
Let $q=2$, $m_1=3$, $m_2=2$ and $D = \{(0,0),(0,1),(1,0),(1,1),(2,0),(2,1)\}$. We have $j=0$. Let 
\[A =
 \begin{pmatrix}
  0 & 1 & 1 & 1 & 1 & 1 \\
  0 & 1 & 1 & 1 & 1 & 1 \\
  0 & 1 & 1 & 1 & 1 & 1 \\
	1 & 1 & 1 & 0 & 1 & 0 \\
	0 & 1 & 1 & 1 & 1 & 1 \\
	1 & 1 & 1 & 0 & 1 & 0 \\
 \end{pmatrix}.
\]
be a $0-1$ matrix indexed by $D \times D$. Let 
$$
X_A = \{(x_k,y_k)_{k=1}^\infty \in \Sigma_{3,2}, \ A((x_k,y_k),(x_{2k},y_{2k}))=1, \ k \geq 1 \}.
$$ 
We look for the solutions $t$ of the systems of equations described in Lemma \ref{arbre}. Using uniqueness we know that
$$
t_{(0,0)} = t_{(0,1)} = t_{(1,0)} = t_{(2,0)}, \ \ \ \ t_{(1,1)} = t_{(2,1)}.
$$
Moreover

$$
t_{(0,0)}^{\gamma q} = \left(t_{(1,0)} + t_{(2,0)}\right)^\gamma + \left(t_{(0,1)} + t_{(1,1)} + t_{(2,1)}\right)^\gamma = 2^\gamma t_{(0,0)}^\gamma + \left(t_{(0,0)} + 2 t_{(1,1)} \right)^\gamma ,
$$

$$
t_{(1,1)}^{\gamma q} = \left(t_{(0,0)} + t_{(1,0)} + t_{(2,0)}\right)^\gamma + t_{(0,1)}^\gamma = (3^\gamma + 1) t_{(0,0)}^\gamma,
$$
thus $t_{(0,0)}^{\gamma q} =  2^\gamma t_{(0,0)}^\gamma + \left(t_{(0,0)} + 2 \left(3^\gamma+1\right)^{\frac{1}{\gamma q}} t_{(0,0)}^{\frac{1}{q}} \right)^\gamma$. Finally we have
$$
t_\varnothing = \left(t_{(1,0)} + t_{(1,0)} + t_{(2,0)}\right)^\gamma + \left(t_{(0,1)} + t_{(1,1)} + t_{(2,1)}\right)^\gamma = 3^\gamma t_{(0,0)}^\gamma + \left(t_{(0,0)} + 2 \left(3^\gamma+1\right)^{\frac{1}{\gamma q}} t_{(0,0)}^{\frac{1}{q}} \right)^\gamma.
$$
Using \textit{Scilab} we get $t_{(0,0)} \simeq 7.1446$, thus $\dim_H(X_A) = \frac{1}{2} \log_2(t_\varnothing) \simeq 1.878$. \newline 
\end{example}

\subsection{The Minkowski dimension of $X_\Omega$} \label{25}

\begin{theorem} \label{test} We have  
\begin{align*}
\begin{split}
\dim_M(X_\Omega) = (q-1)^2 \sum_{p=1}^j \frac{\log_{m_2}(|\textup{Pref}_{0,p}(\Omega)|)}{q^{p+1}} &+ (q-1)(1-q^j \gamma) \sum_{p=j+1}^{\infty} \frac{\log_{m_2}(|\textup{Pref}_{p-j-1,j+1}(\Omega)|)}{q^p}\\ &+ (q-1)(q^{j+1} \gamma-1) \sum_{p=j+1}^{\infty} \frac{\log_{m_2}(|\textup{Pref}_{p-j,j}(\Omega)|)}{q^{p+1}} 
\end{split}
\end{align*}

\end{theorem}

\begin{proof}
Recall that, by definition
$$
\underline{\dim}_M(X_\Omega) = \liminf_{n\rightarrow \infty} \frac{\log_{m_1}(\text{Pref}_{n,L(n)-n}(X_\Omega))}{n}.
$$
We can again fix $\ell \geq j+1$ and take $n = q^\ell r$ with $r \rightarrow \infty$ in this $\liminf$.
Now using the computations used in the proof of Theorem \ref{test2} we get
\begin{align*}
\begin{split}
\log_{m_1}(\text{Pref}_{n,L(n)-n}(X_\Omega)) &\geq \sum_{p=1}^j \# \left\{i \in \left]\frac{L(n)}{q^{p}},\frac{L(n)}{q^{p-1}}\right] : q \nmid i \right\} \log_{m_1}(|\textup{Pref}_{0,p}(\Omega)|) \\& + \sum_{p=j+1}^{\ell} \# \left\{i \in \left]\frac{n}{q^{p-j-1}},\frac{L(n)}{q^{p-1}}\right] : q \nmid i \right\} \log_{m_1}(|\textup{Pref}_{p-j-1,j+1}(\Omega)|) \\& + \sum_{p=j+1}^{\ell} \# \left\{ i \in \left] \frac{L(n)}{q^p},\frac{n}{q^{p-j-1}} \right] : q \nmid i \right\} \log_{m_1}(|\textup{Pref}_{p-j,j}(\Omega)|).
\end{split}
\end{align*}
On the other hand

\begin{align*}
\begin{split}
\log_{m_1}(\text{Pref}_{n,L(n)-n}(X_\Omega)) &\leq \sum_{p=1}^j \# \left\{i \in \left]\frac{L(n)}{q^{p}},\frac{L(n)}{q^{p-1}}\right] : q \nmid i \right\} \log_{m_1}(|\textup{Pref}_{0,p}(\Omega)|) \\& + \sum_{p=j+1}^{\ell} \# \left\{i \in \left]\frac{n}{q^{p-j-1}},\frac{L(n)}{q^{p-1}}\right] : q \nmid i \right\} \log_{m_1}(|\textup{Pref}_{p-j-1,j+1}(\Omega)|) \\& + \sum_{p=j+1}^{\ell} \# \left\{ i \in \left] \frac{L(n)}{q^p},\frac{n}{q^{p-j-1}} \right] : q \nmid i \right\} \log_{m_1}(|\textup{Pref}_{p-j,j}(\Omega)|) \\ &+ \log_{m_1}(m_1 m_2) d_n
\end{split}
\end{align*}
by putting arbitrary digits in the remaining places ($d_n$ being defined in \eqref{nombre}). Remember that $d_n \leq \frac{(\ell+1) L(n)}{q^\ell}+ C \frac{\ell(\ell+1)}{2}$. By letting $r \rightarrow \infty$ we obtain

\begin{align*}
\begin{split}
\underline{\dim}_M(X_\Omega) \geq (q-1)^2 \sum_{p=1}^j \frac{\log_{m_2}(|\textup{Pref}_{0,p}(\Omega)|)}{q^{p+1}} &+ (q-1)(1-q^j \gamma) \sum_{p=j+1}^{\ell} \frac{\log_{m_2}(|\textup{Pref}_{p-j-1,j+1}(\Omega)|)}{q^p}\\ &+ (q-1)(q^{j+1} \gamma-1) \sum_{p=j+1}^{\ell} \frac{\log_{m_2}(|\textup{Pref}_{p-j,j}(\Omega)|)}{q^{p+1}}  
\end{split}
\end{align*}
and

\begin{align*} 
\begin{split}
\overline{\dim}_M(X_\Omega) \leq (q-1)^2 \sum_{p=1}^j \frac{\log_{m_2}(|\textup{Pref}_{0,p}(\Omega)|)}{q^{p+1}} &+ (q-1)(1-q^j \gamma) \sum_{p=j+1}^{\ell} \frac{\log_{m_2}(|\textup{Pref}_{p-j-1,j+1}(\Omega)|)}{q^p}\\ &+ (q-1)(q^{j+1} \gamma-1) \sum_{p=j+1}^{\ell} \frac{\log_{m_2}(|\textup{Pref}_{p-j,j}(\Omega)|)}{q^{p+1}}  \\ &+ \log_{m_2}(m_1 m_2) \frac{\ell+1}{q^\ell}. 
\end{split}
\end{align*}
Since $\ell$ is arbitrary we can conclude.\newline
\end{proof}

\begin{proposition}
We have $\dim_M(X_\Omega) = \dim_H(X_\Omega)$ if and only if the following four conditions are satisfied 
\begin{itemize}
\item the tree $\Gamma_{j}(\Omega)$ is spherically symmetric,
\item $\# \left\{x_1 : (x_1,y_1) y_2 \cdots y_{j+1} \in \textup{Pref}_{1,j}(\Omega) \right\}$ does not depend on $y_1 \cdots y_{j+1} \in \textup{Pref}_{0,j+1}(\Omega)$,
\item for $1 \leq p \leq j$, $\# \left\{y_{p+1} : y_1 \cdots y_{p+1} \in \textup{Pref}_{0,p+1}(\Omega) \right\}$ does not depend on $y_1 \cdots y_p \in \textup{Pref}_{0,p}(\Omega)$,
\item for $p \geq 2$, $\# \left\{x_p : (x_1,y_1) \cdots (x_p,y_p) y_{p+1} \cdots y_{p+j} \in \textup{Pref}_{p,j}(\Omega) \right\}$ does not depend on \\$(x_1,y_1) \cdots (x_{p-1},y_{p-1}) y_p \cdots y_{p+j} \in \textup{Pref}_{p-1,j}(\Omega)$.

\end{itemize}
\end{proposition}

\begin{proof}
Compare the formulas in Theorems \ref{test2} and \ref{test}. We have 
$$
H^\mu_{m_2}(\alpha^2_{p-j} \lor \alpha^1_p) \leq \log_{m_2}(|\text{Pref}_{p-j,j}(\Omega)|),
$$
with equality if and only if every $[u]$ for $u \in \text{Pref}_{p-j,j}(\Omega)$ has equal measure $\mu$, and similar results for $H^\mu_{m_2}(\alpha^1_p)$ and $H^\mu_{m_2}(\alpha^2_{p-j-1} \lor \alpha^1_p)$. Now, the expression of $\mu$ in Proposition \ref{mesure} and uniqueness in Lemma \ref{arbre} give the conditions we stated.\newline

\end{proof}

\section{Generalization to the higher dimensional cases}\label{3}

We are now trying to compute $\dim_H(\mathbb{P}_\mu)$ in any dimension $d \geq 2$. $\Omega$ is now a closed subset of 
$$
\Sigma_{m_1,\ldots,m_d} = (\mathcal{A}_1 \times \cdots \times \mathcal{A}_d)^{\mathbb{N}^*},
$$
where $m_1 \geq \cdots \geq m_d \geq 2$ and $\mathcal{A}_i = \{0,\ldots,m_i-1\}$. We define $$\gamma_i = \frac{\log(m_{i})}{\log(m_{i-1})}$$ and $$L_i : n \in \mathbb{N} \mapsto \left\lceil \frac{n}{\gamma_i} \right\rceil$$ for $2 \leq i \leq d$ ($L_1$ being the identity on $\mathbb{N}$). We can again define the Borel probability measures $\mathbb{P}_\mu$ on $X_\Omega$ as in the two-dimensional case. For $(x^1,\ldots,x^d) \in X_\Omega$ we need to compute $\mathbb{P}_\mu(B_n(x^1,\ldots,x^d))$, where 
$$
B_n(x^1,\ldots,x^d) = \{(u^1,\ldots,u^d) \in \Sigma_{m_1,\dots,m_d} : \forall 1 \leq k \leq d, \ \forall 1 \leq i \leq (L_k \circ \cdots \circ L_1)(n), \ u^k_i = x^k_i \}.
$$

\subsection{Computation of $\dim_H(\mathbb{P}_\mu)$ for $3$-dimensional sponges}

First suppose that $d=3$, as the computation of $\dim_H(\mathbb{P}_\mu)$ in this case helps to better understand the general one. Let $j_2,j_3$ be the unique non-negative integers such that $q^{j_2} \leq \frac{1}{\gamma_2} < q^{j_2+1}$ and $q^{j_3} \leq \frac{1}{\gamma_3} < q^{j_3+1}$. Now we get two cases : either $q^{j_2+j_3} \leq \frac{1}{\gamma_2 \gamma_3} < q^{j_2+j_3+1}$ or $q^{j_2+j_3+1} \leq \frac{1}{\gamma_2 \gamma_3} < q^{j_2+j_3+2}$. Suppose we are in the first one. In this case for all $n$ large enough we have $q^{j_2}n \leq L_2(n) < q^{j_2+1}n$, $q^{j_3}n \leq L_3(n) < q^{j_3+1}n$ and $q^{j_2+j_3}n \leq L_3(L_2(n)) < q^{j_2+j_3+1}n$. In order to compute $\dim_H(\mathbb{P}_\mu)$ we now use the same method as in Proposition $\ref{test2}$. For $n = q^\ell r$ with $\ell$ fixed we can write
$$
\left]\frac{L_3(L_2(n))}{q^\ell},L_3(L_2(n))\right] = \bigsqcup_{p=1}^\ell \left]\frac{L_3(L_2(n))}{q^p},\frac{L_3(L_2(n))}{q^{p-1}}\right].
$$
We now have for all $r$ large enough
\begin{itemize}
\item $1 \leq p \leq j_3 \Longrightarrow \left]\frac{L_3(L_2(n))}{q^p},\frac{L_3(L_2(n))}{q^{p-1}}\right] \subset \left] L_2(n), L_3(L_2(n)) \right]$
\item $j_3+1 \leq p \leq j_3+j_2 \Longrightarrow \left]\frac{L_3(L_2(n))}{q^p},\frac{L_3(L_2(n))}{q^{p-1}}\right] \subset \left] n, L_3(L_2(n)) \right]$. We have $\frac{L_2(n)}{q^{p-j_3-1}} \in \left]\frac{L_3(L_2(n))}{q^p},\frac{L_3(L_2(n))}{q^{p-1}}\right]$ and
$$
i \in \left]\frac{L_3(L_2(n))}{q^p}, \frac{L_2(n)}{q^{p-j_3-1}} \right] \Longrightarrow n < i \leq q^{p-j_3-1} i \leq L_2(n) < q^{p-j_3} i \leq q^{p-1} i \leq L_3(L_2(n)) < q^p i,
$$

$$
i \in \left] \frac{L_2(n)}{q^{p-j_3-1}}, \frac{L_3(L_2(n))}{q^{p-1}} \right] \Longrightarrow n < i \leq q^{p-j_3-2} i \leq L_2(n) < q^{p-j_3-1} i \leq q^{p-1} i \leq L_3(L_2(n)) < q^p i.
$$

\item For $j_3+j_2+1 \leq p \leq \ell$ we have $\frac{L_2(n)}{q^{p-j_3-1}}, \frac{n}{q^{p-j_2-j_3-1}} \in \left]\frac{L_3(L_2(n))}{q^p},\frac{L_3(L_2(n))}{q^{p-1}}\right]$ and $\frac{n}{q^{p-j_2-j_3-1}} \leq \frac{L_2(n)}{q^{p-j_3-1}}$, thus if $i \in \left]\frac{L_3(L_2(n))}{q^p},\frac{n}{q^{p-j_2-j_3-1}} \right]$ then 
$$
q^{p-j_2-j_3-1} i \leq n < q^{p-j_2-j_3} i \leq q^{p-j_3-1} i \leq L_2(n) < q^{p-j_3} i \leq q^{p-1} i \leq L_3(L_2(n)) < q^p i, 
$$

if $i \in \left]\frac{n}{q^{p-j_2-j_3-1}},\frac{L_2(n)}{q^{p-j_3-1}}\right]$ then 
$$
q^{p-j_2-j_3-2} i \leq n < q^{p-j_2-j_3-1} i \leq q^{p-j_3-1} i \leq L_2(n) < q^{p-j_3} i \leq q^{p-1} i \leq L_3(L_2(n)) < q^p i,
$$

and if $i \in \left]\frac{L_2(n)}{q^{p-j_3-1}}, \frac{L_3(L_2(n))}{q^{p-1}}\right]$ then 
$$
q^{p-j_2-j_3-2} i \leq n < q^{p-j_2-j_3-1} i \leq q^{p-j_3-2} i \leq L_2(n) < q^{p-j_3-1} i \leq q^{p-1} i \leq L_3(L_2(n)) < q^p i.
$$
\end{itemize}
Denote by $\alpha^3_p$, $\alpha^2_p$ and $\alpha^1_p$ the partitions of $\Omega$ into cylinders of length $p$ along all three coordinates, the second and the third ones, and the third one respectively. Using the same approach as in the two dimensional case we can get 
\begin{align*}
\begin{split}
\dim_H(\mathbb{P}_\mu) &= (q-1)^2 \sum_{p=1}^{j_3} \frac{H^\mu_{m_3}(\alpha^1_p)}{q^{p+1}} + (q-1)(\gamma_3 q^{j_3+1}-1) \sum_{p=j_3+1}^{j_2+j_3} \frac{H^\mu_{m_3}(\alpha^2_{p-j_3} \lor \alpha^1_p)}{q^{p+1}} \\&+ (q-1)(1-\gamma_3 q^{j_3}) \sum_{p=j_3+1}^{j_2+j_3} \frac{H^\mu_{m_3}(\alpha^2_{p-j_3-1} \lor \alpha^1_p)}{q^{p}} \\&+ (q-1)(\gamma_2 \gamma_3 q^{j_2+j_3+1}-1) \sum_{p=j_2+j_3+1}^\infty \frac{H^\mu_{m_3}(\alpha^3_{p-j_2-j_3} \lor \alpha^2_{p-j_3} \lor \alpha^1_p)}{q^{p+1}} \\&+ (q-1)(\gamma_3 q^{j_3} - \gamma_2 \gamma_3 q^{j_2 + j_3}) \sum_{p=j_2+j_3+1}^\infty \frac{H^\mu_{m_3}(\alpha^3_{p-j_2-j_3-1} \lor \alpha^2_{p-j_3} \lor \alpha^1_p)}{q^{p}} \\&+ (q-1) (1-\gamma_3 q^{j_3}) \sum_{p=j_2+j_3+1}^\infty \frac{H^\mu_{m_3}(\alpha^3_{p-j_2-j_3-1} \lor \alpha^2_{p-j_3-1} \lor \alpha^1_p)}{q^{p}}.
\end{split}
\end{align*}
If we suppose now that $q^{j_2+j_3+1} \leq \frac{1}{\gamma_2 \gamma_3} < q^{j_2+j_3+2}$, we have $\frac{L_2(n)}{q^{p-j_3-1}} \leq \frac{n}{q^{p-j_2-j_3-2}}$ for $n$ large enough and we get 
\begin{align} \label{uuu}
\begin{split}
\dim_H(\mathbb{P}_\mu) &= (q-1)^2 \sum_{p=1}^{j_3} \frac{H^\mu_{m_3}(\alpha^1_p)}{q^{p+1}} + (q-1)(\gamma_3 q^{j_3+1}-1) \sum_{p=j_3+1}^{j_2+j_3+1} \frac{H^\mu_{m_3}(\alpha^2_{p-j_3} \lor \alpha^1_p)}{q^{p+1}} \\&+ (q-1)(1-\gamma_3 q^{j_3}) \sum_{p=j_3+1}^{j_2+j_3+1} \frac{H^\mu_{m_3}(\alpha^2_{p-j_3-1} \lor \alpha^1_p)}{q^{p}} \\&+ (q-1)(\gamma_3 q^{j_3+1}-1) \sum_{p=j_2+j_3+2}^\infty \frac{H^\mu_{m_3}(\alpha^3_{p-j_2-j_3-1} \lor  \alpha^2_{p-j_3} \lor \alpha^1_p)}{q^{p+1}} \\&+ (q-1)(\gamma_2 \gamma_3 q^{j_2+j_3+1} - \gamma_3 q^{j_3}) \sum_{p=j_2+j_3+2}^\infty \frac{H^\mu_{m_3}(\alpha^3_{p-j_2-j_3-1} \lor  \alpha^2_{p-j_3-1} \lor \alpha^1_p)}{q^{p}} \\&+ (q-1) (1- \gamma_2 \gamma_3 q^{j_2+j_3+1}) \sum_{p=j_2+j_3+2}^\infty \frac{H^\mu_{m_3}(\alpha^3_{p-j_2-j_3-2} \lor  \alpha^2_{p-j_3-1} \lor \alpha^1_p)}{q^{p}}.
\end{split}
\end{align}
In the next subsection we will adopt a more general point of view to avoid this dichotomy case.

\subsection{Results in any dimension} \label{32}

We get back to the general case, by first introducing some notations and making a few observations before stating the theorems. Let $I \subset \mathbb{N}^*$ and $K \subset \llbracket 1,d \rrbracket$ be finite sets. If $x \in \Omega$ and $(x^k_i)_{\substack{i \in I \\ k \in K}}$ is a finite set of coordinates of $x$ (the upper index corresponding to the ``geometric'' coordinate and the lower one being the digit) we define the generalized cylinder
$$
\left[(x^k_i)_{\substack{i \in I \\ k \in K}}\right] = \{y \in \Omega : y^k_i = x^k_i \ \ \forall i \in I, \ \forall k \in K\}.
$$
For some arbitrary coordinate functions $\chi_1,\ldots,\chi_N \in \{\{x \in \Omega \mapsto x^k_i\} : k \in \llbracket 1,d \rrbracket, \ i \geq 1\}$ we also define $$\text{Pref}_{\chi_1,\ldots,\chi_N}(\Omega) = \{\left(\chi_1(x),\ldots,\chi_N(x)\right) : x \in \Omega\}.$$ For all $t \in \llbracket 2,d \rrbracket$, let $j_t \in \mathbb{N}$ such that
$$
q^{j_t} \leq \frac{1}{\gamma_{t}} < q^{j_t+1}.
$$
There is a unique sequence of integers $(n_{t})_{2 \leq t \leq d}$ such that 
$$
\forall t \in \llbracket 1,d-1 \rrbracket, \ q^{j_d+j_{d-1}+\cdots+j_{t+1}+n_{t+1}} \leq \frac{1}{\gamma_d \gamma_{d-1} \cdots \gamma_{t+1}} < q^{j_d+j_{d-1}+\cdots+j_{t+1}+n_{t+1}+1}.
$$
Let $$p_t = j_d+j_{d-1}+\cdots+j_{t+1}+n_{t+1}.$$ The sequence $(n_t)$ takes its values in $\llbracket 0,d-2 \rrbracket$ and is non-decreasing; moreover $n_d = 0$ and $n_{t} \in \{n_{t+1},n_{t+1}+1\}$ for $2 \leq t \leq d-1$. The integers $j_t$, $t \in \llbracket 2,d \rrbracket$ and $n_t$, $t \in \llbracket 2,d-1 \rrbracket$ are the $2d-3$ parameters mentioned in the introduction. Thus we get that for all $n$ large enough, for $s \in \llbracket 1,d-1 \rrbracket$
$$
\forall t \in \llbracket s,d-1 \rrbracket, \ \forall p \in \llbracket p_{s}+1, p_{s-1} \rrbracket, \ \frac{L_{t} \circ \cdots \circ L_1(n)}{q^{p-p_t-1}} \in \left] \frac{L_d \circ \cdots \circ L_1(n)}{q^p}, \frac{L_d \circ \cdots \circ L_1(n)}{q^{p-1}} \right]
$$
and 
$$
\frac{L_d \circ \cdots \circ L_1(n)}{q^p} \geq L_{s-1} \circ \cdots \circ L_1(n),
$$
with $p_0 = \ell$ and $L_0(n) = 0$. If $p \in \llbracket 1, p_{d-1} \rrbracket$ then
$$
\frac{L_d \circ \cdots \circ L_1(n)}{q^p} \geq L_{d-1} \circ \cdots \circ L_1(n).
$$
For $s \in \llbracket 1,d-1 \rrbracket$ let $\sigma_s \in \frak{S}(\llbracket s,d-1 \rrbracket)$ be the unique permutation such that the sequence
$$
\left(\frac{L_{\sigma_s(t)} \circ \cdots \circ L_1(n)}{q^{p-p_{\sigma_s(t)}-1}}\right)_{t \in \llbracket s,d-1 \rrbracket}
$$
is non-decreasing for all $n$ large enough and all $p$. We define 
$$
I^{s,s-1}_p = \left]\frac{L_d \circ \cdots \circ L_1(n)}{q^p}, \frac{L_{\sigma_s(s)} \circ \cdots \circ L_1(n)}{q^{p-p_{\sigma_s(s)}-1}} \right],
$$

$$
I^{s,t}_p = \left]\frac{L_{\sigma_s(t)} \circ \cdots \circ L_1(n)}{q^{p-p_{\sigma_s(t)}-1}} , \frac{L_{\sigma_s(t+1)} \circ \cdots \circ L_1(n)}{q^{p-p_{\sigma_s(t+1)}-1}} \right]
$$
for $t \in \llbracket s,d-2 \rrbracket$ and
$$
I^{s,d-1}_p = \left]\frac{L_{\sigma_s(d-1)} \circ \cdots \circ L_1(n)}{q^{p-p_{\sigma_s(d-1)}-1}} , \frac{L_d \circ \cdots \circ L_1(n)}{q^{p-1}} \right]. 
$$
We will use the partitions

$$
\left] \frac{L_d \circ \cdots \circ L_1(n)}{q^p}, \frac{L_d \circ \cdots \circ L_1(n)}{q^{p-1}} \right] = \bigsqcup_{t=s-1}^{d-1} I^{s,t}_p
$$
for all $p \in \llbracket p_{s}+1, p_{s-1} \rrbracket$. Observe that for $i \in \left] \frac{L_d \circ \cdots \circ L_1(n)}{q^p}, \frac{L_d \circ \cdots \circ L_1(n)}{q^{p-1}} \right]$ such that $q \nmid i$ we have $$q^{p-p_k-2} i \leq L_{k} \circ \cdots \circ L_1(n) < q^{p-p_k} i$$ for all $k \in \llbracket 1,d-1 \rrbracket$. Hence for $k \in \llbracket 1,d-1 \rrbracket$ either 

$$
q^{p-p_k-2} i \leq L_{k} \circ \cdots \circ L_1(n) < q^{p-p_k-1} i
$$
or

$$
q^{p-p_k-1} i \leq L_{k} \circ \cdots \circ L_1(n) < q^{p-p_k} i.
$$
Moreover if $i \in I_p^{s,t}$ then
\begin{align*}
\begin{split}
&\frac{L_{\sigma_s(s)} \circ \cdots \circ L_1(n)}{q^{p-p_{\sigma_s(s)}-1}} \leq \cdots \leq \frac{L_{\sigma_s(t)} \circ \cdots \circ L_1(n)}{q^{p-p_{\sigma_s(t)}-1}} < i \leq \\& \frac{L_{\sigma_s(t+1)} \circ \cdots \circ L_1(n)}{q^{p-p_{\sigma_s(t+1)}-1}} \leq \cdots \leq \frac{L_{\sigma_s(d-1)} \circ \cdots \circ L_1(n)}{q^{p-p_{\sigma_s(d-1)}-1}}.
\end{split}
\end{align*}
For $s \in \llbracket 1,d-1 \rrbracket$ and $t \in \llbracket s-1,d-1 \rrbracket$ let $(p^{s,t}_k)_{k \in \llbracket s,d-1 \rrbracket}$ be defined by $p^{s,t}_{k} = p_k+1$ if $k \in \sigma_s(\llbracket s,t \rrbracket)$, and $p^{s,t}_k = p_k$ otherwise. Then we have
$$
i \in I_p^{s,t} \Longrightarrow \forall \ k \in \llbracket s,d-1 \rrbracket, \ q^{p-p^{s,t}_{k}-1} i \leq L_{k} \circ \cdots \circ L_1(n) < q^{p-p^{s,t}_{k}} i.
$$
Thus the $\mathbb{P}_\mu$-mass of an arbitrary ``quasi-cube'' is
\begin{align} \label{masse}
\begin{split}
\mathbb{P}_\mu(B_n(x^1, \ldots, x^d)) = &\bigg(\prod_{p=1}^{j_d} \prod_{\substack{i \in \left]\frac{L_d \circ \cdots \circ L_1(n)}{q^p}, \frac{L_d \circ \cdots \circ L_1(n)}{q^{p-1}} \right] \\ q \nmid i}} \mu \left( \left[x_i^d \cdots x^d_{q^{p-1}i} \right] \right) \bigg) \\&\boldsymbol{\cdot} \ \prod_{s=2}^{d-1} \ \bigg( \prod_{p=p_s+1}^{p_{s-1}} \ \prod_{t=s-1}^{d-1} \ \prod_{\substack{i \in I^{s,t}_p \\ q \nmid i}} \mu (C^{s,t}_{p,i}(x)) \bigg) \\ &\boldsymbol{\cdot} \bigg(\prod_{p=p_{1}+1}^{\ell} \ \prod_{t=0}^{d-1} \ \prod_{\substack{i \in I^{1,t}_p \\ q \nmid i}} \mu (C^{1,t}_{p,i}(x)) \bigg) \boldsymbol{\cdot} D_n(x^1,\ldots,x^d),
\end{split}
\end{align} 
where 
\begin{align*}
\begin{split}
 C^{s,t}_{p,i}(x) =&\bigg[ \bigg(x_i^{s},\ldots,x_i^d\bigg) \cdots \bigg(x_{q^{p-p_s^{s,t}-1} i}^{s},\ldots,x_{q^{p-p_s^{s,t}-1} i}^d\bigg) \bigg] \\&\bigcap \bigg[\bigg(x_{q^{p-p_s^{s,t}} i}^{s+1},\ldots,x_{q^{p-p_{s}^{s,t}} i}^d\bigg) \cdots \bigg(x_{q^{p-p_{s+1}^{s,t}-1} i}^{s+1},\ldots,x_{q^{p-p_{s+1}^{s,t}-1} i}^d\bigg) \bigg] \bigcap \cdots \\&\bigcap \bigg[ \bigg(x^{d-1}_{q^{p-p_{d-2}^{s,t}} i}, x^{d}_{q^{p-p_{d-2}^{s,t}} i}\bigg) \cdots \bigg(x^{d-1}_{q^{p-p_{d-1}^{s,t}-1} i}, x^{d}_{q^{p-p_{d-1}^{s,t}-1} i}\bigg) \bigg] \bigcap \bigg[ x^d_{q^{p-p_{d-1}^{s,t}} i} \ \cdots \ x^d_{q^{p-1}i} \bigg]
\end{split}
\end{align*}
and $D_n(x^1,\ldots,x^d)$ is the residual term. Note that $C^{s,t}_{p,i}(x)$ can also be compactly written as
\begin{align*}
\begin{split}
&\left[\pi^s(x)_i \cdots \pi^s(x)_{q^{p-p_s^{s,t}-1} i}\right] \bigcap \left[\pi^{s+1}(x)_{q^{p-p_s^{s,t}} i} \cdots \pi^{s+1}(x)_{q^{p-p_{s+1}^{s,t}-1} i}\right] \bigcap \cdots \\&\bigcap \left[\pi^d(x)_{q^{p-p_{d-1}^{s,t}} i} \cdots \pi^d(x)_{q^{p-1}i}\right],
\end{split}
\end{align*}
using the projections $\pi^k : x \mapsto (x^k,\ldots,x^d)$ for $k \in \llbracket 1,d \rrbracket$.\newline 

Now, for all $p \geq 1$ we define

\begin{align*}
\begin{split}
\delta^{d,d-1}_{p} &= \lim_{n \rightarrow \infty} \frac{\# \left\{i \in \left]\frac{L_d \circ \cdots \circ L_1(n)}{q^p}, \frac{L_d \circ \cdots \circ L_1(n)}{q^{p-1}} \right] : q \nmid i \right\}}{L_d \circ \cdots \circ L_1(n)} = \frac{(q-1)^2}{q^{p+1}},\\ \delta^{s,s-1}_{p} &= \lim_{n \rightarrow \infty} \frac{\# \left\{i \in I^{s,0}_p : q \nmid i \right\}}{L_d \circ \cdots \circ L_1(n)} = \frac{(q^{p_{\sigma_s(s)}+1} \prod_{i=\sigma_s(s)+1}^d \gamma_i -1)(q-1)}{q^{p+1}}, \\ \delta^{s,t}_p &= \lim_{n \rightarrow \infty} \frac{\# \left\{i \in I^{s,t}_p : q \nmid i \right\}}{L_d \circ \cdots \circ L_1(n)} \\ \\ &= \frac{(q^{p_{\sigma_s(t+1)}} \prod_{i=\sigma_s(t+1)+1}^d \gamma_i -q^{p_{\sigma_s(t)}} \prod_{i=\sigma_s(t)+1}^d \gamma_i)(q-1)}{q^p} \ \text{for} \ t \in \llbracket s, d-2 \rrbracket
\end{split}
\end{align*}
and 

$$
\delta^{s,d-1}_{p} = \lim_{n \rightarrow \infty} \frac{\# \left\{i \in I^{s,s}_p : q \nmid i \right\}}{L_d \circ \cdots \circ L_1(n)} = \frac{(1-q^{p_{\sigma_s(d-1)}} \prod_{i=\sigma_s(d-1)+1}^d \gamma_i)(q-1)}{q^p}.\newline
$$
Moreover denote by $\alpha^k_p$ the partition of $\Omega$ into cylinders of length $p$ along the last $k$ coordinates for $k \in \llbracket 1,d \rrbracket$. Finally let

$$
\widetilde{H}_{s,p}^\mu = \sum_{t=s-1}^{d-1} \delta^{s,t}_p H^\mu_{m_d} \left(\alpha^{1}_p \lor \alpha^2_{p-p^{s,t}_{d-1}} \lor \alpha^3_{p-p^{s,t}_{d-2}} \lor \cdots \lor \alpha^{d-s+1}_{p-p^{s,t}_{s}}\right)
$$
for $s \in \llbracket 1,d \rrbracket$.\\ 

\begin{theorem}

The Borel probability measure $\mathbb{P}_\mu$ is exact dimensional and its dimension is
$$
S(\Omega,\mu) = \left(\sum_{p=1}^{j_d} \widetilde{H}_{d,p}^\mu \right) + \left(\sum_{s=2}^{d-1} \ \sum_{p=p_{s}+1}^{p_{s-1}} \widetilde{H}_{s,p}^\mu \right) + \sum_{p=p_{1}+1}^\infty \widetilde{H}_{1,p}^\mu.
$$
\end{theorem}

\begin{proof}
We use exactly the same method as in the proof of Theorem \ref{test2}, using the computation of $\mathbb{P}_\mu(B_n(x^1,\ldots,x^d))$ above, the different families of i.i.d random variables $$\left\{ Y^{s,t}_{p,i} : x \in X_\Omega \mapsto -\log(\mu(C^{s,t}_{p,i}(x)))\right\}_{i\in I_p^{s,t}}$$ whose expectations are $H^\mu_{m_d} \left(\alpha^{1}_p \lor \alpha^2_{p-p^{s,t}_{d-1}} \lor \alpha^3_{p-p^{s,t}_{d-2}} \lor \cdots \lor \alpha^{d-s+1}_{p-p^{s,t}_{s}}\right)$ respectively, and Theorem \ref{calme} and Lemma \ref{proba} repeatedly. We then show again that the residual term $D_n(x^1,\ldots,x^d)$, which is larger than or equal to the $\mathbb{P}_\mu$-mass of those points in $X_\Omega$ which share the same symbolic coordinates as $x$ for those indices $j$ which do not appear in the cylinders of the forme $C^{s,t}_{p,i}(x)$ with $p \leq \ell$, is $\mathbb{P}_\mu$-almost always negligible. To this end, we use like in the proof of Theorem \ref{test2} Borel--Cantelli lemma and the set
$$
S_n = \{(x^1,\ldots,x^d) \in X_\Omega : D_n(x^1,\ldots,x^d) \leq (2 m_1 m_2 \cdots m_d)^{-d_n}\},
$$
where the exponent 
$$
d_n = L_d \circ \cdots \circ L_1(n) - \sum_{p=1}^{\ell} \# \left\{ i \in \mathbb{N} \cap \left] \frac{L_d \circ \cdots \circ L_1(n)}{q^p},\frac{L_d \circ \cdots \circ L_1(n)}{q^{p-1}} \right] : q \nmid i \right\} p
$$
can likewise easily be controlled.
\end{proof}

We can again optimize this quantity following the method we used in the two-dimensional case, by conditioning all the entropy terms appearing in the third part of this expression for $p \geq p_{1}+2$ by the finest partition appearing in the term $\widetilde{H}_{1,p_{1}+1}^\mu$. We know that for all $s$ we have 
\begin{align} \label{ordre}
\begin{split}
&t \leq t' \Longrightarrow \forall k, \ p_k^{s,t} \leq p_k^{s,t'},\\&t < t' \Longrightarrow \exists k, \ p_k^{s,t} < p_k^{s,t'},
\end{split}
\end{align}
so this partition is the one appearing in the $t=0$ term, i.e. 
\begin{align*}
\begin{split}
\alpha &= \alpha^1_{p_{1}+1} \lor \alpha^2_{p_{1}+1-p_{d-1}^{1,0}} \lor \alpha^3_{p_{1}+1-p_{d-2}^{1,0}} \lor \cdots \lor \alpha^d_{p_{1}+1-p_{1}^{1,0}}.\\ &= \alpha^1_{p_{1}+1} \lor \alpha^2_{p_{1}+1-p_{d-1}} \lor \alpha^3_{p_{1}+1-p_{d-2}} \lor \cdots \lor \alpha^d_{1}.
\end{split}
\end{align*}
If $C$ is a cylinder of this partition in $\Omega$, denote by $\Omega_C$, $\theta_C$ and $\mu_C$ the associate rooted set at $C \in \alpha$ in $\Omega$, its $\mu$-mass and the normalized measure induced on it respectively. Since for $p \geq p_{1}+2$ and $t \in \llbracket 0,d-1 \rrbracket$ we can write 
\begin{align*}
\begin{split}
&H^\mu_{m_d} \left(\alpha^{1}_p \lor \alpha^2_{p-p^{1,t}_{d-1}} \lor \alpha^3_{p-p^{1,t}_{d-2}} \lor \cdots \lor \alpha^{d}_{p-p^{1,t}_{1}}\right) \\&= H^\mu_{m_d}(\alpha) + H^\mu_{m_d} \left(\alpha^{1}_p \lor \alpha^2_{p-p^{1,t}_{d-1}} \lor \alpha^3_{p-p^{1,t}_{d-2}} \lor \cdots \lor \alpha^{d}_{p-p^{1,t}_{1}} | \alpha \right)\\ &= H^{\mu}_{m_d}(\alpha) + \sum_{C \in \mathcal{C}} \theta_C H^{\mu_C}_{m_d} \bigg(\alpha^{1}_{p-1} \lor \alpha^2_{p-p^{1,t}_{d-1}-1} \lor \alpha^3_{p-p^{1,t}_{d-2}-1} \lor \cdots \lor \alpha^{d}_{p-p^{1,t}_{1}-1}(\Omega_C) \bigg),
\end{split}
\end{align*}
we get 
\begin{align} \label{S}
\begin{split}
S(\Omega,\mu) &= \sum_{p=1}^{j_d} \widetilde{H}_{d,p}^\mu + \sum_{s=2}^{d-1} \ \sum_{p=p_s+1}^{p_{s-1}} \widetilde{H}_{s,p}^\mu \\ &+ \sum_{t=1}^{d-1} \delta^{1,t}_{p_{1}+1} H^\mu_{m_d} \left(\alpha^1_{p_{1}+1} \lor \alpha^2_{p_{1}+1-p_{d-1}^{1,t}} \lor \cdots \lor \alpha^d_{p_{1}+1-p_{1}^{1,t}} \right) \\ &+ \left(\delta_{p_{1}+1}^{1,0}+\sum_{p=p_{1}+2}^\infty \ \sum_{t=0}^{d-1} \delta_p^{1,t}\right) H^\mu_{m_d}(\alpha) + \frac{1}{q} \sum_{C \in \mathcal{C}} \theta_C S(\Omega_C,\mu_C).
\end{split}
\end{align}
Now we can obtain the unique optimal measure as in the proof of Theorem \ref{mesure} by getting the $q_C$ with a recursive reasoning and repeating the argument for the entire suitable graphs. To make things clearer and to highlight the fact that the structure of the optimal measure is similar to the one appearing in the two-dimensional case, we introduce now the unique sequence of coordinate functions $(\chi_i)_{i \geq 1}$ such that if we reorder the partitions of $\Omega$ appearing in the expression of $\dim_H({\mathbb{P}_\mu})$ above as an increasing sequence $\beta_1 \leq \beta_2 \leq \cdots$ (the symbol $\leq$ corresponding there to the ``finer than'' partial order) we have 
$$
H^\mu(\beta_i) = -\int_{\Omega} \log_{m_d} \left(\mu \left( \left[\chi_1(x) \cdots \chi_i(x)\right]\right)\right) d \mu(x)
$$
for all $i \geq 1$. Here we used a slight generalization of the notion of cylinders we defined at the beginning of Section \ref{32}, allowing ourselves to use any family $A \subset \mathbb{N}^* \times \llbracket 1,d \rrbracket$ of coordinates of $x$ and not necessarily a product. This order is exactly the following (using again facts \eqref{ordre}):
\begin{align*}
\begin{split}
&\alpha^1_1 \leq \cdots \leq \alpha^1_{p_{d-1}} \leq \alpha^1_{p_{d-1}+1} \lor \alpha^2_{p_{d-1}+1-p_{d-1}^{d-1,d-1}} \leq \alpha^1_{p_{d-1}+1} \lor \alpha^2_{p_{d-1}+1-p_{d-1}^{d-1,d-2}} \leq \\&\alpha^1_{p_{d-1}+2} \lor \alpha^2_{p_{d-1}+2-p_{d-1}^{d-1,d-1}} \leq \alpha^1_{p_{d-1}+2} \lor \alpha^2_{p_{d-1}+2-p_{d-1}^{d-1,d-2}} \leq \cdots \leq \alpha^1_{p_{d-2}} \lor \alpha^2_{p_{d-2}-p_{d-1}^{d-1,d-1}} \leq \\&\alpha^1_{p_{d-2}} \lor \alpha^2_{p_{d-2}-p_{d-1}^{d-1,d-2}} \leq \cdots \leq \alpha^1_{p_2+1} \lor \alpha^{2}_{p_2+1-p_{d-1}^{2,d-1}} \lor \cdots \lor \alpha^{d-1}_{p_2+1-p_2^{2,d-1}} \leq \cdots \leq \\&\alpha^1_{p_2+1} \lor \alpha^{2}_{p_2+1-p_{d-1}^{2,1}} \lor \cdots \lor \alpha^{d-1}_{p_2+1-p_2^{2,1}} \leq \cdots \leq \alpha^1_{p_1} \lor \alpha^{2}_{p_1-p_{d-1}^{2,d-1}} \lor \cdots \lor \alpha^{d-1}_{p_1-p_2^{2,d-1}} \leq \cdots \leq \\&\alpha^1_{p_1} \lor \alpha^{2}_{p_1-p_{d-1}^{2,1}} \lor \cdots \lor \alpha^{d-1}_{p_1-p_2^{2,1}} \leq \alpha^1_{p_1+1} \lor \alpha^2_{p_1+1-p_{d-1}^{1,d-1}} \lor \cdots \lor \alpha^d_{p_1+1-p_1^{1,d-1}} \leq \cdots \leq \\&\alpha^1_{p_1+1} \lor \alpha^2_{p_1+1-p_{d-1}^{1,0}} \lor \cdots \lor \alpha^d_{p_1+1-p_1^{1,0}} = \alpha \leq \alpha^1_{p_1+2} \lor \alpha^2_{p_1+2-p_{d-1}^{1,d-1}} \lor \cdots \lor \alpha^d_{p_1+2-p_1^{1,d-1}} \leq \cdots \leq \\&\alpha^1_{p_1+2} \lor \alpha^2_{p_1+2-p_{d-1}^{1,0}} \lor \cdots \lor \alpha^d_{p_1+2-p_1^{1,0}} \leq \cdots
\end{split}
\end{align*}
For example, when $d=3$ and $\dim_H({\mathbb{P}_\mu})$ is given by \eqref{uuu}, this sequence is given by 
$$
(\chi_i)_{i \geq 1} = \left(x^3_1, \ldots, x^3_{j_3},x^3_{j_3+1},x^2_1,x^3_{j_3+2},x^2_2,\ldots,x^3_{j_2+j_3+2},x^1_1,x^2_{j_2+2},\ldots \right).
$$
We also denote by $(\delta_i)_{i \in \llbracket 1,N \rrbracket}$ the sequence of real factors giving weights to the $N$ entropies in $S(\Omega,\mu)$ (see \eqref{S}) when being reordered that way. Let $$N = p_{1} + 1 + \sum_{k=1}^{d-1} (p_{1} + 1 - p_k)$$ be the number of coordinates $\chi_i$ appearing in the partition $\alpha$ distinguished above. Finally for $(X_1,\ldots,X_N) \in \text{Pref}_{\chi_1,\ldots,\chi_N}(\Omega)$ let $\Gamma_{(X_1,\ldots,X_N)}(\Omega)$ be the directed graph whose set of vertices is $(X_1,\ldots,X_N) \cup \bigcup_{\ell=1}^\infty \text{Pref}_{\chi_1,\ldots,\chi_{N+\ell d}}(\Omega)$, and where for all $\ell \geq 0$ there is a directed edge from $u = X_1 \cdots X_{N+\ell d}$ to another one $v$ if and only if $v = X_1 \cdots X_{N+\ell d} X_{N+\ell d +1} \cdots X_{N+(\ell+1)d}$ for some $X_i$, $i \in \llbracket N+\ell d + 1, N+(\ell+1)d \rrbracket$.\\

\begin{theorem}
Let $\omega_1 = \sum_{i=1}^N \delta_i$, $\omega_k = \frac{\sum_{i=k}^N \delta_i}{\sum_{i=k-1}^N \delta_i}$ for $2 \leq k \leq N$ and $\omega_{N+1} = \frac{1}{q \delta_N}$. For all $(X_1,\ldots,X_N) \in \textup{Pref}_{\chi_1,\ldots,\chi_N}(\Omega)$ there is a unique vector $t \in \left[1,m_d^{\omega_{N+1} d}\right]^{\Gamma_{(X_1,\ldots,X_N)}(\Omega)}$ such that for all $\ell \geq 0$ and $\left(X_1,\ldots,X_{N+\ell d}\right) \in \Gamma_{(X_1,\ldots,X_N)}(\Omega)$ we have
\small
$$
\left(t_{X_1 \cdots X_{N+\ell d}}\right)^{\frac{1}{\tilde{\omega}_{N-d+1} \omega_{N+1}}} = \sum_{X'_{N+\ell d+1}} \bigg( \sum_{X'_{N+\ell d+2}} \bigg( \cdots \bigg( \sum_{X'_{N+(\ell+1)d}} t_{X_1 \cdots X'_{N+(\ell+1)d}} \bigg)^{\omega_N} \cdots \bigg)^{\omega_{N-d+3}} \bigg)^{\omega_{N-d+2}},
$$
\normalsize
where $\tilde{\omega}_{N-d+1} = \omega_1 \omega_2 \cdots \omega_{N-d+1}$. Moreover if we define 
$$
t_\varnothing = \sum_{X'_1} \bigg( \sum_{X'_2} \bigg( \cdots \bigg( \sum_{X'_{N-1}} \bigg( \sum_{X'_N} t_{X'_1 \cdots X'_N} \bigg)^{\omega_N} \bigg)^{\omega_{N-1}} \cdots \bigg)^{\omega_3} \bigg)^{\omega_2},
$$
the unique Borel probability measure maximizing $S(\Omega,\mu)$ is defined for all $\ell \geq 0$ by
\small
\begin{align}\label{optimal}
\begin{split}
&\mu \left( \left[X_1 \cdots X_{N+\ell d}\right]\right) \\&= \frac{t_{X_1 \cdots X_N}}{t_\varnothing} \prod_{p=2}^{N} \bigg(\sum_{X'_{p}} \bigg( \sum_{X'_{p+1}} \bigg( \cdots \bigg(\sum_{X'_{N-1}} \bigg(\sum_{X'_N} t_{X_1 \cdots X_{p-1} X'_p \cdots X'_N} \bigg)^{\omega_N} \bigg)^{\omega_{N-1}} \cdots \bigg)^{\omega_{p+2}} \bigg)^{\omega_{p+1}} \bigg)^{\omega_{p}-1} \\&\boldsymbol{\cdot} \prod_{k=1}^{\ell} t_{X_1 \cdots X_{N+kd}} \ t_{X_1 \cdots X_{N+(k-1)d}}^{-\frac{1}{\tilde{\omega}_{N-d+1} \omega_{N+1}}} \ \prod_{p=2}^d\bigg(\sum_{X'_{N+(k-1)d+p}} \bigg( \cdots \bigg(\sum_{X'_{N+kd}} t_{X_1 \cdots X'_{N+kd}} \bigg)^{\omega_N} \cdots \bigg)^{\omega_{N-d+p+1}} \bigg)^{\omega_{N-d+p}-1} 
\end{split}
\end{align}
\normalsize
and its Hausdorff dimension is equal to $\omega_1 \log_{m_d}(t_\varnothing)$.
\end{theorem}

\begin{proof}

The existence and uniqueness of $t$ are checked using a fixed point theorem as in Lemma \ref{arbre}. We get with these notations that 
\begin{align*}
\begin{split}
S(\Omega,\mu) &= \sum_{i=1}^N \delta_i H^\mu_{m_d}(\beta_i) + \frac{1}{q} \sum_{X_1,\ldots,X_N} \theta_{X_1 \cdots X_N} S(\Omega_{X_1 \cdots X_N},\mu_{X_1 \cdots X_N})\\&= \omega_1 \bigg( H^\mu_{m_d}(\beta_1) + \omega_2 \sum_{X_1} \theta_{X_1} \bigg( -\sum_{X_2} \frac{\theta_{X_1 X_2}}{\theta_{X_1}} \log_{m_d} \left(\frac{\theta_{X_1 X_2}}{\theta_{X_1}} \right) \\&+ \omega_3 \sum_{X_2} \frac{\theta_{X_1 X_2}}{\theta_{X_1}} \bigg(-\sum_{X_3} \frac{\theta_{X_1 X_2 X_3}}{\theta_{X_1 X_2}} \log_{m_d} \left(\frac{\theta_{X_1 X_2 X_3}}{\theta_{X_1 X_2}} \right) \\&+ \omega_4 \sum_{X_3} \frac{\theta_{X_1 X_2 X_3}}{\theta_{X_1 X_2}} \bigg( \cdots + \omega_{N} \sum_{X_{N-1}} \frac{\theta_{X_1 \cdots X_{N-1}}}{\theta_{X_1 \cdots X_{N-2}}} \bigg( -\sum_{X_{N}} \frac{\theta_{X_1 \cdots X_{N}}}{\theta_{X_1 \cdots X_{N-1}}} \log_{m_d} \left(\frac{\theta_{X_1 \cdots X_{N}}}{\theta_{X_1 \cdots X_{n-1}}} \right) \\&+ \omega_{N+1} \sum_{X_N} \frac{\theta_{X_1 \cdots X_{N}}}{\theta_{X_1 \cdots X_{N-1}}} S \left(\Omega_{X_1 \cdots X_N},\mu_{X_1 \cdots X_N} \right) \bigg) \cdots \bigg) \bigg) \bigg) \bigg).
\end{split}
\end{align*}
Optimizing this expression as before, we get that $\theta_{X_1 \cdots X_N}$ equals
$$
\frac{z_{X_1 \cdots X_N}}{z_\varnothing} \prod_{p=2}^{N} \bigg(\sum_{X'_{p}} \bigg( \sum_{X'_{p+1}} \bigg( \cdots \bigg(\sum_{X'_{N-1}} \bigg(\sum_{X'_N} z_{X_1 \cdots X_{p-1} X'_p \cdots X'_N} \bigg)^{\omega_N} \bigg)^{\omega_{N-1}} \cdots \bigg)^{\omega_{p+2}} \bigg)^{\omega_{p+1}} \bigg)^{\omega_{p}-1}
$$
where $z_{X_1 \cdots X_N} = m_d^{\omega_{N+1} S \left(\Omega_{X_1 \cdots X_N} \right)}$ and $z_\varnothing = m_d^{\frac{S(\Omega)}{\omega_1}}$. It remains to optimize the conditional measures on the subtrees $\Omega_{X_1 \cdots X_N}$, by maximizing the expression $S \left(\Omega_{X_1 \cdots X_N},\mu_{X_1 \cdots X_N} \right)$ which is equal to
\begin{align*}
\begin{split}
&\sum_{i=N-d+1}^N \delta_i H^{\mu_{X_1 \cdots X_N}}_{m_d} \left(\beta_i \left(\Omega_{X_1 \cdots X_N} \right) \right) + \frac{1}{q} \sum_{X_{N+1},\ldots,X_{N+d}} \frac{\theta_{X_1 \cdots X_{N+d}}}{\theta_{X_1 \cdots X_N}} S \left( \Omega_{X_1 \cdots X_{N+d}},\mu_{X_1 \cdots X_{N+d}} \right) \\&=\widetilde{\omega}_{N-d+1} \bigg(-\sum_{X_{N+1}} \frac{\theta_{X_1 \cdots X_{N+1}}}{\theta_{X_1 \cdots X_N}} \log_{m_d} \bigg(\frac{\theta_{X_1 \cdots X_{N+1}}}{\theta_{X_1 \cdots X_N}} \bigg) \\&+ \omega_{N-d+2} \sum_{X_{N+1}} \frac{\theta_{X_1 \cdots X_{N+1}}}{\theta_{X_1 \cdots X_N}} \bigg( \cdots \\&+ \omega_{N} \sum_{X_{N+d-1}} \frac{\theta_{X_1 \cdots X_{N+d-1}}}{\theta_{X_1 \cdots X_{N+d-2}}} \bigg( -\sum_{X_{N+d}} \frac{\theta_{X_1 \cdots X_{N+d}}}{\theta_{X_1 \cdots X_{N+d-1}}} \log_{m_d} \left(\frac{\theta_{X_1 \cdots X_{N+d}}}{\theta_{X_1 \cdots X_{N+d-1}}} \right) \\&+ \omega_{N+1} \sum_{X_{N+d}} \frac{\theta_{X_1 \cdots X_{N+d}}}{\theta_{X_1 \cdots X_{N+d-1}}} S\left(\Omega_{X_1 \cdots X_{N+d}},\mu_{X_1 \cdots X_{N+d}}\right) \bigg) \cdots \bigg) \bigg) 
\end{split}
\end{align*}
and repeating the argument for the entire graphs. This yields the desired results.

\end{proof}

\begin{theorem}
Let $\mu$ be the Borel probability measure on $\Omega$ defined in the last theorem, and let $\mathbb{P}_\mu$ be the corresponding probability measure on $X_\Omega$. Let $x \in X_\Omega$. Then 
$$
\liminf_{n \rightarrow \infty} \frac{-\log_{m_d}(\mathbb{P}_\mu(B_n(x)))}{L_d \circ \cdots \circ L_1(n)} \leq \omega_1 \log_{m_d}(t_\varnothing).
$$
Using Theorem \ref{calme} we deduce that 
$$
\dim_H(X_\Omega) = \omega_1 \log_{m_d}(t_\varnothing) = \frac{q-1}{q} \log_{m_d}(t_\varnothing).
$$
\end{theorem}

\begin{proof}

Let $\Lambda(n) =  \bigcup\limits_{k=1}^d \left\{\frac{L_k \circ \cdots \circ L_1(n)}{q^r} : r \in \mathbb{N} \right\}$. We can reorder the elements of $\Lambda(n)$ as the following increasing sequence :\newpage
\begin{align*}
\begin{split}
&L_d \circ \cdots \circ L_1(n) \geq  \frac{L_d \circ \cdots \circ L_1(n)}{q} \geq \cdots \geq \frac{L_d \circ \cdots \circ L_1(n)}{q^{p_{d-1}}} \geq L_{d-1} \circ \cdots \circ L_1(n) \geq \\& \frac{L_d \circ \cdots \circ L_1(n)}{q^{p_{d-1}+1}} \geq
 \frac{L_{d-1} \circ \cdots \circ L_1(n)}{q} \geq \cdots \geq \frac{L_{d-1} \circ \cdots \circ L_1(n)}{q^{p_{d-2}-p_{d-1}-1}} \geq \frac{L_d \circ \cdots \circ L_1(n)}{q^{p_{d-2}}} \geq \\& \frac{L_{\sigma_{d-2}(d-1)} \circ \cdots \circ L_1(n)}{q^{p_{d-2}-p_{\sigma_{d-2}(d-1)}}} \geq 
\frac{L_{\sigma_{d-2}(d-2)} \circ \cdots \circ L_1(n)}{q^{p_{d-2}-p_{\sigma_{d-2}(d-2)}}} \geq \frac{L_d \circ \cdots \circ L_1(n)}{q^{p_{d-2}+1}} \geq \cdots \geq \\&  \frac{L_{\sigma_{d-2}(d-1)} \circ \cdots \circ L_1(n)}{q^{p_{d-3}-p_{\sigma_{d-2}(d-1)}-1}} \geq \frac{L_{\sigma_{d-2}(d-2)} \circ \cdots \circ L_1(n)}{q^{p_{d-3}-p_{\sigma_{d-2}(d-2)}-1}} \geq 
\frac{L_d \circ \cdots \circ L_1(n)}{q^{p_{d-3}}} \geq \cdots \geq \\& \frac{L_d \circ \cdots \circ L_1(n)}{q^{p_{1}}} \geq \frac{L_{\sigma_{1}(d-1)} \circ \cdots \circ L_1(n)}{q^{p_{1}-p_{\sigma_{1}(d-1)}}} \geq \cdots \geq \frac{L_{\sigma_{1}(1)} \circ \cdots \circ L_1(n)}{q^{p_{1}-p_{\sigma_{1}(1)}}} \geq 
 \frac{L_d \circ \cdots \circ L_1(n)}{q^{p_{1}+1}} \geq \\& \frac{L_{\sigma_{1}(d-1)} \circ \cdots \circ L_1(n)}{q^{p_{1}-p_{\sigma_{1}(d-1)}+1}} \geq \cdots \geq \frac{L_{\sigma_{1}(1)} \circ \cdots \circ L_1(n)}{q^{p_{1}-p_{\sigma_{1}(1)}+1}} \geq \frac{L_d \circ \cdots \circ L_1(n)}{q^{p_{1}+2}} \geq \cdots
\end{split}
\end{align*}
We denote by $$\phi_0(n) = L_d \circ \cdots \circ L_1(n) \geq \phi_1(n) \geq \phi_2(n) \geq \ldots$$ this sequence, which is valid for all $n$. Observe that $\phi_N(n) = \frac{L_d \circ \cdots \circ L_1(n)}{q^{p_{1}+1}}$. We now fix $n \geq 1$. Let $S \geq 0$ be the unique integer such that we have
$$
\phi_S(n) < 1 \leq \phi_{S-1}(n) \leq \cdots \leq \phi_0(n).
$$
We can write $S = N + Md +R$, with $M \geq 0$ and $R \in \llbracket 0, d-1 \rrbracket$. Recall formula \eqref{masse}. With these notations we get that
\begin{equation} \label{formule}
\mathbb{P}_\mu(B_n(x)) = \prod_{k=1}^S \ \prod_{\substack{\phi_k(n) < i \leq \phi_{k-1}(n) \\ q \nmid i}} \mu \left( \left[ \chi_1 \left(x|_{J_i} \right) \cdots \chi_k\left(x|_{J_i} \right) \right] \right).
\end{equation}
Now, for all $1 \leq k \leq N-1$, we have
\begin{align} \label{liste2}
\begin{split}
&\mu \left( \left[ \chi_1 \left(x|_{J_i} \right) \cdots \chi_k\left(x|_{J_i} \right) \right] \right) \\&= \frac{1}{t_\varnothing} \prod_{p=2}^k \bigg( \sum_{\chi_p\left(x|_{J_i}\right)'} \cdots \bigg(\sum_{\chi_N \left(x|_{J_i}\right)'} t_{\chi_1 \left(x|_{J_i}\right) \cdots \chi_{p-1}\left(x|_{J_i}\right) \chi_p\left(x|_{J_i}\right)' \cdots \chi_N \left(x|_{J_i}\right)'} \bigg)^{\omega_N} \cdots \bigg)^{\omega_p-1} \\&\boldsymbol{\cdot} \bigg(\sum_{\chi_{k+1}(x|_{J_i})'} \cdots \bigg( \sum_{\chi_N(x|_{J_i})'} t_{\chi_1(x|_{J_i}) \cdots \chi_{k}(x|_{J_i}) \chi_{k+1}(x|_{J_i})' \cdots \chi_N(x|_{J_i})'} \bigg)^{\omega_N} \cdots \bigg)^{\omega_{k+1}}.
\end{split}
\end{align}
If $R=0$ then 
\begin{align*}  
\begin{split}
&\mu \left( \left[ \chi_1 \left(x|_{J_i} \right) \cdots \chi_{N+Md}\left(x|_{J_i} \right) \right] \right) \\&= \mu \left( \left[ \chi_1 \left(x|_{J_i} \right) \cdots \chi_N\left(x|_{J_i} \right) \right] \right) \boldsymbol{\cdot} \prod_{r=1}^{M} \bigg[t_{\chi_1 \left(x|_{J_i}\right) \cdots \chi_{N+rd}\left(x|_{J_i}\right)} \ t_{\chi_1\left(x|_{J_i}\right) \cdots \chi_{N+(r-1)d}\left(x|_{J_i}\right)}^{-\frac{1}{\tilde{\omega}_{N-d+1} \omega_{N+1}}} \\ &\prod_{p=2}^d\bigg(\sum_{\chi_{N+(r-1)d+p}\left(x|_{J_i}\right)'} \bigg( \cdots \bigg(\sum_{\chi_{N+rd}\left(x|_{J_i}\right)'} t_{\chi_1\left(x|_{J_i}\right) \cdots \chi_{N+rd}\left(x|_{J_i}\right)'} \bigg)^{\omega_N} \cdots \bigg)^{\omega_{N-d+p+1}} \bigg)^{\omega_{N-d+p}-1}\bigg],
\end{split}
\end{align*}
and if $R \in \llbracket{1,d-1}\rrbracket$ then 
\small
\begin{align} \label{liste}
\begin{split}
&\mu \left( \left[ \chi_1 \left(x|_{J_i} \right) \cdots \chi_{N+Md+R}\left(x|_{J_i} \right) \right] \right) \\&= \mu \left( \left[ \chi_1 \left(x|_{J_i} \right) \cdots \chi_N\left(x|_{J_i} \right) \right] \right) \boldsymbol{\cdot} \prod_{r=1}^{M} \bigg[t_{\chi_1 \left(x|_{J_i}\right) \cdots \chi_{N+rd}\left(x|_{J_i}\right)} \ t_{\chi_1\left(x|_{J_i}\right) \cdots \chi_{N+(r-1)d}\left(x|_{J_i}\right)}^{-\frac{1}{\tilde{\omega}_{N-d+1} \omega_{N+1}}} \\ &\prod_{p=2}^d\bigg(\sum_{\chi_{N+(r-1)d+p}\left(x|_{J_i}\right)'} \bigg( \cdots \bigg(\sum_{\chi_{N+rd}\left(x|_{J_i}\right)'} t_{\chi_1\left(x|_{J_i}\right) \cdots \chi_{N+rd}\left(x|_{J_i}\right)'} \bigg)^{\omega_N} \cdots \bigg)^{\omega_{N-d+p+1}} \bigg)^{\omega_{N-d+p}-1}\bigg] \\& \boldsymbol{\cdot} t_{\chi_1\left(x|_{J_i}\right) \ldots \chi_{N+Md}\left(x|_{J_i}\right)}^{-({\tilde{\omega}_{N-d+1} \omega_{N+1}})^{-1}} \\& \boldsymbol{\cdot} \prod_{p=2}^{R} \bigg(\sum_{\chi_{N+Md+p}\left(x|_{J_i}\right)'} \bigg( \cdots \bigg(\sum_{\chi_{N+(M+1)d}\left(x|_{J_i}\right)'} t_{\chi_1\left(x|_{J_i}\right) \cdots \chi_{N+(M+1)d}\left(x|_{J_i}\right)'} \bigg)^{\omega_N} \cdots \bigg)^{\omega_{N-d+p+1}} \bigg)^{\omega_{N-d+p}-1} \\&\boldsymbol{\cdot} \bigg(\sum_{\chi_{N+Md+R+1}(x|_{J_i})'} \cdots \bigg( \sum_{\chi_{N+(M+1)d}(x|_{J_i})'} t_{\chi_1(x|_{J_i}) \cdots \chi_{N+(M+1)d}(x|_{J_i})'} \bigg)^{\omega_N} \cdots \bigg)^{\omega_{N-d+R+1}}.
\end{split}
\end{align}
\normalsize
Observe that for $k \in \llbracket 0,d \rrbracket$ and $r \in \mathbb{N}$ we have $\phi_{N-k+rd}(n) = \frac{\phi_{N-k}(n)}{q^r}$. Thus for $k \in \llbracket 0,d-1 \rrbracket$ and $r \in \mathbb{N}$
\begin{equation} \label{phi}
\phi_{N-k+rd}(n) < i \leq \phi_{N-k+rd-1}(n) \Longleftrightarrow \phi_{N-k}(n) < q^r i \leq \phi_{N-k-1}(n).
\end{equation}
Now we can develop the expression \eqref{formule} of $\mathbb{P}_\mu(B_n(x))$ and group together the terms with the same number of sums. We get $t_\varnothing^{-1}$ for all $1 \leq i \leq \phi_0(n)$ such that $q \nmid i$; using property \eqref{phi} we get
$$\prod\limits_{r=0}^M \ \prod\limits_{\substack{1 \leq i \leq \phi_{N+rd-1}(n) \\ q \nmid i}} t_{\chi_1(x|_{J_i}) \cdots \chi_{N+rd}(x|_{J_i})} = \prod\limits_{\substack{\kappa = q^r i \leq \phi_{N-1}(n) \\ q \nmid i}} t_{\chi_1(x|_{J_i}) \cdots \chi_{N+rd}(x|_{J_i})},$$ and $$\prod\limits_{r=0}^M \ \prod\limits_{\substack{1 \leq i \leq \phi_{N+rd}(n) \\ q \nmid i}} t_{\chi_1(x|_{J_i}) \cdots \chi_{N+rd}(x|_{J_i})}^{-(\tilde{\omega}_{N-d+1} \omega_{N+1})^{-1}} = \prod\limits_{\substack{\kappa = q^r i \leq \phi_{N}(n) \\ q \nmid i}} t_{\chi_1(x|_{J_i}) \cdots \chi_{N+rd}(x|_{J_i})}^{-(\tilde{\omega}_{N-d+1} \omega_{N+1})^{-1}};$$ 
we gather the product of terms coming from \eqref{liste2} with $k \in \llbracket 1,N-d \rrbracket$ to get
\begin{align*}
\begin{split}
&\prod\limits_{k=1}^{N-d} \prod\limits_{\phi_k(n) < i \leq \phi_{k-1}(n)} \bigg(\sum_{\chi_{k+1}(x|_{J_i})'} \cdots \bigg( \sum_{\chi_N(x|_{J_i})'} t_{\chi_1(x|_{J_i}) \cdots \chi_{k}(x|_{J_i}) \chi_{k+1}(x|_{J_i})' \cdots \chi_N(x|_{J_i})'} \bigg)^{\omega_N} \cdots \bigg)^{\omega_{k+1}} \\&= \prod_{p=d+1}^{N} \ \prod_{\substack{\phi_{N-p+1}(n) < i \leq \phi_{N-p}(n) \\ q \nmid i}} \ \bigg(\sum_{\chi_{N-p+2}(x|_{J_i})'} \cdots \bigg(\sum_{\chi_{N}(x|_{J_i})'} t_{\chi_1(x|_{J_i}) \cdots \chi_{N}(x|_{J_i})'} \bigg)^{\omega_N} \cdots \bigg)^{\omega_{N-p+2}};
\end{split}
\end{align*}
and similarly
\small
\begin{align*}
\begin{split}
&\prod\limits_{k=N-d+1}^{N-1} \prod\limits_{\phi_k(n) < i \leq \phi_{k-1}(n)} \bigg(\sum_{\chi_{k+1}(x|_{J_i})'} \cdots \bigg( \sum_{\chi_N(x|_{J_i})'} t_{\chi_1(x|_{J_i}) \cdots \chi_{k}(x|_{J_i}) \chi_{k+1}(x|_{J_i})' \cdots \chi_N(x|_{J_i})'} \bigg)^{\omega_N} \cdots \bigg)^{\omega_{k+1}} \\&= \prod\limits_{p=2}^{d} \ \prod\limits_{\substack{\phi_{N-p+1}(n) < i \leq \phi_{N-p}(n) \\ q \nmid i}} \ \bigg(\sum_{\chi_{N-p+2}(x|_{J_i})'} \cdots \bigg(\sum_{\chi_{N}(x|_{J_i})'} t_{\chi_1(x|_{J_i}) \cdots \chi_{N}(x|_{J_i})'} \bigg)^{\omega_N} \cdots \bigg)^{\omega_{N-p+2}},
\end{split}
\end{align*}
\normalsize
that we combine with 
\scalefont{0.78}\selectfont
\begin{align*}
\begin{split}
&\prod\limits_{k=0}^M \prod\limits_{R=1}^{d-1} \prod\limits_{\substack{\phi_{N+kd+R}(n) < i \leq \phi_{N+kd+R-1}(n) \\ q \nmid i}} \bigg(\sum\limits_{\chi_{N+kd+R+1}(x|_{J_i})'} \cdots \bigg( \sum\limits_{\chi_{N+(k+1)d}(x|_{J_i})'} t_{\chi_1(x|_{J_i}) \cdots \chi_{N+(M+1)d}(x|_{J_i})'} \bigg)^{\omega_N} \cdots \bigg)^{\omega_{N-d+R+1}}
\\&= \prod\limits_{r=1}^{M+1} \prod\limits_{p=2}^d \prod\limits_{\substack{\phi_{N+rd-p+1}(n) < i \leq \phi_{N+rd-p}(n) \\ q \nmid i}} \bigg(\sum_{\chi_{N+rd-p+2}(x|_{J_i})'} \cdots \bigg(\sum_{\chi_{N+rd}(x|_{J_i})'} t_{\chi_1(x|_{J_i}) \cdots \chi_{N+rd}(x|_{J_i})'} \bigg)^{\omega_N} \cdots \bigg)^{\omega_{N-p+2}},
\end{split}
\end{align*}
\normalsize
to get 
\small
$$
\prod_{p=2}^{d} \ \prod_{\substack{\phi_{N-p+1}(n) < \kappa = q^r i \leq \phi_{N-p}(n) \\ q \nmid i}} \ \bigg(\sum_{\chi_{N+rd-p+2}(x|_{J_i})'} \cdots \bigg(\sum_{\chi_{N+rd}(x|_{J_i})'} t_{\chi_1(x|_{J_i}) \cdots \chi_{N+rd}(x|_{J_i})'} \bigg)^{\omega_N} \cdots \bigg)^{\omega_{N-p+2}};
$$
\normalsize
finally we combine in a similar way all the remaining terms from the products \eqref{liste} and \eqref{liste2} and obtain
\begin{align*}
\begin{split}                                              
&\prod_{p=2}^{d} \ \prod_{\substack{\kappa = q^r i \leq \phi_{N-p+1}(n)\\ q \nmid i}} \ \bigg(\sum_{\chi_{N+rd-p+2}(x|_{J_i})'} \cdots \bigg(\sum_{\chi_{N+rd}(x|_{J_i})'} t_{\chi_1(x|_{J_i}) \cdots \chi_{N+rd}(x|_{J_i})'} \bigg)^{\omega_N} \cdots \bigg)^{\omega_{N-p+2}-1} \\& \boldsymbol{\cdot} \prod_{p=d+1}^{N} \ \prod_{\substack{i \leq \phi_{N-p+1}(n) \\ q \nmid i}} \ \bigg(\sum_{\chi_{N-p+2}(x|_{J_i})'} \cdots \bigg(\sum_{\chi_{N}(x|_{J_i})'} t_{\chi_1(x|_{J_i}) \cdots \chi_{N}(x|_{J_i})'} \bigg)^{\omega_N} \cdots \bigg)^{\omega_{N-p+2}-1}.
\end{split}
\end{align*}
Thus
\small
\begin{align*}
\begin{split}
&\mathbb{P}_\mu(B_n(x)) \\&= t_\varnothing^{-\# \{i \in \llbracket 1,\phi_0(n)\rrbracket, \ q \nmid i\}} \bigg(\prod_{\substack{\kappa = q^r i \leq \phi_{N-1}(n) \\ q \nmid i}} t_{\chi_1(x|_{J_i}) \cdots \chi_{N+rd}(x|_{J_i})} \bigg) \bigg(\prod_{\substack{\kappa = q^r i \leq \phi_{N}(n) \\ q \nmid i}} t_{\chi_1(x|_{J_i}) \cdots \chi_{N+rd}(x|_{J_i})}^{-(\tilde{\omega}_{N-d+1} \omega_{N+1})^{-1}} \bigg) \\& \boldsymbol{\cdot} \prod_{p=2}^{d} \ \prod_{\substack{\kappa = q^r i \leq \phi_{N-p+1}(n)\\ q \nmid i}} \ \bigg(\sum_{\chi_{N+rd-p+2}(x|_{J_i})'} \cdots \bigg(\sum_{\chi_{N+rd}(x|_{J_i})'} t_{\chi_1(x|_{J_i}) \cdots \chi_{N+rd}(x|_{J_i})'} \bigg)^{\omega_N} \cdots \bigg)^{\omega_{N-p+2}-1} \\& \boldsymbol{\cdot} \prod_{p=2}^{d} \ \prod_{\substack{\phi_{N-p+1}(n) < \kappa = q^r i \leq \phi_{N-p}(n) \\ q \nmid i}} \ \bigg(\sum_{\chi_{N+rd-p+2}(x|_{J_i})'} \cdots \bigg(\sum_{\chi_{N+rd}(x|_{J_i})'} t_{\chi_1(x|_{J_i}) \cdots \chi_{N+rd}(x|_{J_i})'} \bigg)^{\omega_N} \cdots \bigg)^{\omega_{N-p+2}} \\& \boldsymbol{\cdot} \prod_{p=d+1}^{N} \ \prod_{\substack{i \leq \phi_{N-p+1}(n) \\ q \nmid i}} \ \bigg(\sum_{\chi_{N-p+2}(x|_{J_i})'} \cdots \bigg(\sum_{\chi_{N}(x|_{J_i})'} t_{\chi_1(x|_{J_i}) \cdots \chi_{N}(x|_{J_i})'} \bigg)^{\omega_N} \cdots \bigg)^{\omega_{N-p+2}-1} \\& \boldsymbol{\cdot} \prod_{p=d+1}^{N} \ \prod_{\substack{\phi_{N-p+1}(n) < i \leq \phi_{N-p}(n) \\ q \nmid i}} \ \bigg(\sum_{\chi_{N-p+2}(x|_{J_i})'} \cdots \bigg(\sum_{\chi_{N}(x|_{J_i})'} t_{\chi_1(x|_{J_i}) \cdots \chi_{N}(x|_{J_i})'} \bigg)^{\omega_N} \cdots \bigg)^{\omega_{N-p+2}}.
\end{split}
\end{align*}
\normalsize
For $\kappa = q^r i$ with $q \nmid i$, let
\begin{align*}
\begin{split}
R_1(\kappa) &= \log_{m_d}\left(t_{\chi_1(x|_{J_i}) \cdots \chi_{N+rd}(x|_{J_i})}\right),\\ R_2(\kappa) &= \log_{m_d} \bigg(\sum_{\chi_{N+rd}(x|_{J_i})'} t_{\chi_1(x|_{J_i}) \cdots \chi_{N+rd}(x|_{J_i})'}\bigg),
\end{split}
\end{align*}
and for $p \in \llbracket 3,d \rrbracket$, let
$$
R_{p}(\kappa) = \log_{m_d} \bigg(\sum_{\chi_{N+rd-p+2}(x|_{J_i})'} \bigg( \cdots \bigg(\sum_{\chi_{N+rd}(x|_{J_i})'} t_{X_1(x|_{J_i}) \cdots \chi_{N+rd}(x|_{J_i})'} \bigg)^{\omega_N} \cdots \bigg)^{\omega_{N-p+3}}\bigg).
$$
For $p \in \llbracket 1,d \rrbracket$ and $n \geq 1$ let
$$
u^p_n = \frac{1}{n} \sum_{\kappa=1}^n R_p(\kappa) 
$$
and for $p \in \llbracket d+1,N \rrbracket$ let
$$
u^p_n = \frac{1}{n} \sum_{i \leq n, \ q \nmid i} \log_{m_d} \bigg(\sum_{\chi_{N-p+2}(x|_{J_i})'} \bigg( \cdots \bigg(\sum_{\chi_{N}(x|_{J_i})'} t_{\chi_1(x|_{J_i}) \cdots \chi_{N}(x|_{J_i})'} \bigg)^{\omega_N} \cdots \bigg)^{\omega_{N-p+3}} \bigg).
$$
This gives us $N$ bounded sequences. We can now write 
\begin{align*}
\begin{split}
-\log_{m_d}(\mathbb{P}_\mu(B_n(x))) &= (\tilde{\omega}_{N-d+1} \omega_{N+1})^{-1} \lfloor \phi_N(n) \rfloor u^1_{\lfloor \phi_N(n) \rfloor} -  \lfloor\phi_{N-1}(n) \rfloor u^1_{\lfloor \phi_{N-1}(n)\rfloor} \\&+ \sum_{k=0}^{N-2} \left(\lfloor \phi_{k+1}(n)\rfloor u^{N-k}_{\lfloor\phi_{k+1}(n)\rfloor}-\omega_{k+2} \lfloor\phi_k(n)\rfloor u^{N-k}_{\lfloor\phi_k(n)\rfloor}\right) \\ &+ \# \{i \in \llbracket 1,\phi_0(n) \rrbracket \} \log_{m_d}(t_\varnothing).
\end{split}
\end{align*}
Furthermore some basic recursive computations give us the values of the exponents
\begin{align*}
\begin{split}
\omega_1 &= \sum_{p=1}^{j_d} \delta_p^{d,d-1} + \sum_{s=2}^{d-1} \sum_{p=p_s+1}^{p_{s-1}} \sum_{t=s-1}^{d-1} \delta_p^{s,t} + \sum_{p=p_1+1}^\infty \sum_{t=0}^{d-1} \delta_p^{1,t} = \frac{q-1}{q} \\&= \lim\limits_{n \rightarrow \infty} \frac{\# \left\{i \leq L_d \circ \cdots \circ L_1(n) : q \nmid i \right\}}{L_d \circ \cdots \circ L_1(n)},
\end{split}
\end{align*}
$$
\omega_2 = \frac{\omega_1-\delta_1}{\omega_1} = \lim\limits_{n \rightarrow \infty} \frac{\# \left\{i \leq \frac{L_d \circ \cdots \circ L_1(n)}{q} : q \nmid i \right\}}{\# \left\{i \leq L_d \circ \cdots \circ L_1(n) : q \nmid i \right\}} = \frac{1}{q}, \ \ \omega_3 = \frac{\omega_1 - \delta_1 - \delta_2}{\omega_1-\delta_1} = \frac{1}{q}, 
$$
$$
\omega_4 = \cdots = \omega_{p_{d-1}+2} = \frac{1}{q}, \ \ \omega_{p_{d-1}+3} = \gamma_d q^{p_{d-1}},
$$
$$
\omega_{p_{d-1}+4} = \frac{1}{\gamma_d q^{p_{d-1}+1}}, \ \ \ldots \ , \ \omega_k = \frac{\omega_1 - \sum\limits_{i=1}^{k-1} \delta_i}{\omega_1 - \sum\limits_{i=1}^{k-2} \delta_i} = \lim\limits_{n \rightarrow \infty} \frac{\# \left\{i \leq \phi_{k-1}(n) : q \nmid i \right\}}{\# \left\{i \leq \phi_{k-2}(n) : q \nmid i \right\}}, \ \ \ldots
$$
$$
\omega_N = q^{p_{\sigma_1(1)}-p_{\sigma_1(2)}} \frac{\prod_{i=\sigma_1(1)+1}^d \gamma_i}{\prod_{i=\sigma_1(2)+1}^d \gamma_i}, \ \ \omega_{N+1} = \frac{q^{p_1-p_{\sigma_1(1)}}}{(q-1) \prod_{\sigma_1(1)+1}^d \gamma_i}.
$$
This yields $$(\tilde{\omega}_{N-d+1} \omega_{N+1})^{-1} = q^{p_{\sigma_1(1)}+1} \prod_{i=\sigma_1(1)+1}^d \gamma_i$$ and the asymptotic equivalences $$\lfloor\phi_{N-1}(n)\rfloor \sim (\tilde{\omega}_{N-d+1} \omega_{N+1})^{-1} \lfloor \phi_N(n)\rfloor,$$
$$
\lfloor\phi_{k+1}(n)\rfloor \sim \omega_{k+2} \lfloor\phi_k(n)\rfloor
$$ 
for all $k \in \llbracket 0,N-2 \rrbracket$, when $n \rightarrow +\infty$. We conclude by using again lemma \ref{combi}.

\end{proof}

\begin{theorem}

For $n_1,\ldots,n_d \in \mathbb{N}$ let
$$
\textup{Pref}_{n_1,\ldots,n_d}(\Omega) = \left\{ u \in \prod_{i=1}^d (\llbracket 0,m_i-1 \rrbracket \times \cdots \times \llbracket 0,m_d-1 \rrbracket)^{n_i} : \Omega \cap \left[u\right] \neq \varnothing\right\}.
$$
We have
\begin{align*}
\begin{split}
\dim_M(X_\Omega) &= \sum_{p=1}^{j_d} \delta_p^{d,d-1} |\textup{Pref}_{0,\ldots,0,p}(\Omega)| \\&+ \sum_{s=2}^{d-1} \ \sum_{p=p_s+1}^{p_{s-1}} \ \sum_{t=s-1}^{d-1} \delta_p^{s,t} |\textup{Pref}_{0,\ldots,0,p-p_s^{s,t},p_s^{s,t}-p_{s+1}^{s,t},\ldots,p_{d-2}^{s,t}-p_{d-1}^{s,t},p_{d-1}^{s,t}}(\Omega)| \\&+ \sum_{p=p_1+1}^\infty \ \sum_{t=0}^{d-1} \delta^{1,t}_p |\textup{Pref}_{p-p_1^{1,t},p_1^{1,t}-p_{2}^{1,t},\ldots,p_{d-2}^{1,t}-p_{d-1}^{1,t},p_{d-1}^{1,t}}(\Omega)|.
\end{split}
\end{align*}
\end{theorem}

\begin{proof}
The proof follows the same path as in the two-dimensional case. We leave it to the reader, along with the characterization of the equality case with the Hausdorff dimension.
\end{proof}
\newpage

\begin{appendix}

\pagenumbering{Roman}
\setcounter{page}{1}

\section{}\nonumber

\begin{lemma} \label{convexe}
Let $p_1,\ldots,p_m \geq 0$ with $\sum_{i=1}^m p_1 = 1$, and let $q_1,\ldots,q_m \in \mathbb{R}$. Then
$$
\sum_{i=1}^m p_i (-\log(p_i)+q_i) \leq \log \left( \sum_{i=1}^m e^{q_i} \right),
$$
with equality if and only if $p_i = \frac{e^{q_i}}{\sum_{j=1}^m e^{q_j}}$ for all $i$.
\end{lemma}

\begin{proof}
See \cite[Corollary 1.5]{ref3}.\newline
\end{proof}

\begin{lemma} \label{proba}
Let $(\Omega, \mathcal{F}, \mathbb{P})$ be a probability space, $(m_n) \in (\mathbb{N}^*)^{\mathbb{N}^*}$ be a strictly increasing sequence such that $\sum_{n=1}^\infty \frac{1}{m_n^2} < +\infty$ and for all $n \geq 1$ let $(X_{i,n})_{i \in \llbracket 1,m_n \rrbracket}$ be a family of independent centered random variables on $(\Omega, \mathcal{F}, \mathbb{P})$. Assume that there exists $K \geq 0$ such that
$$
\forall n \in \mathbb{N}^*, \ \forall i \in \llbracket 1,m_n \rrbracket, \ \mathbb{E} \left[X_{i,n}^4 \right] \leq K.
$$
Then $\frac{1}{m_n} \sum_{i=1}^{m_n} X_{i,n} \xrightarrow[n \rightarrow \infty]{\text{a.s}} 0$.
\end{lemma}

\begin{proof}
Fix $n \geq 1$. We have
\begin{align*}
\begin{split}
\mathbb{E} \left[ \left(\sum_{i=1}^{m_n} X_{i,n} \right)^4 \right] &= \mathbb{E} \left[ \sum_{i=1}^{m_n} X_{i,n}^4 + 6 \sum_{i<j} X_{i,n}^2 X_{j,n}^2 \right] \\ &\leq m_n K + 3 m_n (m_n-1) K \\ &\leq 3 K m_n^2
\end{split}
\end{align*}
by using independence and Jensen's inequality. Now $\sum_{n=1}^\infty \left( \frac{1}{m_n} \sum_{i=1}^{m_n} X_{i,n} \right)^4$ is a well-defined random variable taking values in $\mathbb{R}^+ \cup \{+\infty\}$. Moreover by the monotone convergence theorem 
$$
\mathbb{E} \left[\sum_{n=1}^\infty \left( \frac{1}{m_n} \sum_{i=1}^{m_n} X_{i,n} \right)^4 \right] = \sum_{n=1}^\infty \frac{1}{m_n^4} \mathbb{E} \left[ \left(\sum_{i=1}^{m_n} X_{i,n} \right)^4 \right] \leq 3K  \sum_{n=1}^\infty \frac{1}{m_n^2} < +\infty.
$$

Thus $\sum_{n=1}^\infty \left( \frac{1}{m_n} \sum_{i=1}^{m_n} X_{i,n} \right)^4 < +\infty$ a.s and $\frac{1}{m_n} \sum_{i=1}^{m_n} X_{i,n} \xrightarrow[n \rightarrow \infty]{\text{a.s}} 0$. \newline

\end{proof}

\begin{lemma} \label{combi}
Let $p \in \mathbb{N}^*$ and for $1 \leq j \leq p$ let $(u^j_n) \in \mathbb{R}^\mathbb{N}$ be $p$ bounded sequences with
$$
\lim_{n \rightarrow \infty} u^j_{n+1} - u^j_n = 0.
$$
For $j \in \llbracket 1,p \rrbracket$ let $\phi_j, \psi_j : \mathbb{N} \rightarrow \mathbb{N}$ be such that
$$
\exists c_j,r_j > 0, \ \exists A_j,B_j \in \mathbb{N}, \ \forall n, \ |\phi_j(n) - \lceil r_j n \rceil| \leq A_j \ \text{and} \ |\psi_j(n) - \lceil c_j n \rceil| \leq B_j.
$$
Then we have
$$
\liminf_{n \rightarrow \infty} \sum_{j=1}^p \left(u^j_{ \phi_j (n) } - u^j_{ \psi_j (n) } \right) \leq 0.
$$
\end{lemma}

\begin{proof}
Observe that for all $j$ and $n$ we have $\lceil r_j n \rceil \in \{\phi_j(n) + k, \ |k| \leq A_j\} $ and $\lceil c_j n \rceil \in \{\psi_j(n) + k, \ |k| \leq B_j\}$. Thus 
$$
|u^j_{\phi_j(n)}-u^j_{\lceil r_j n \rceil}| \leq \max_{|k| \leq A_j} |u^j_{\phi_j(n)}-u^j_{\phi_j(n)+k}| \underset{n\to \infty}{\longrightarrow} 0
$$
using the hypothesis on $u^j$ above. Similarly $|u^j_{\psi_j(n)}-u^j_{\lceil c_j n \rceil}| \underset{n\to \infty}{\longrightarrow} 0$. Now conclude with \cite[Lemma 5.4]{ref5} or \cite[Lemma 4.1]{ref10}.\newline
\end{proof}

\begin{lemma} \label{loc}
Let $\mu$ be a Borel probability measure on $\Sigma_{m_1,m_2}$. Suppose that $\mu$ is exact dimensional with respect to the metric 
$$
\tilde{d}((x_k,y_k)_{k=1}^\infty,(u_k,v_k)_{k=1}^\infty) = e^{-\min \left\{k \geq 1, \ (x_k,y_k) \neq (u_k,v_k)\right\}},
$$
with dimension $\delta$. Denote by $\delta_2$ the lower Hausdorff dimension of $\pi_* \mu$ with respect to the metric induced by $\tilde{d}$, and let $\underline{\delta_1}$ and $\overline{\delta_1}$ be the essential infimum and the essential supremum of the lower Hausdorff dimensions of the conditional measures $\mu^y$ with respect to $\tilde{d}$ again, where $\mu_y$ is obtained from the disintegration of $\mu$ with respect to $\pi_* \mu$. Then, with respect to the metric $d$, for $\mu$-almost every point $z$ we have
$$
\frac{\overline{\delta_1}}{\log(m_1)} + \frac{\delta_2}{\log(m_2)} \leq \underline{\dim}_{\text{loc}}(\mu,z) \leq \overline{\dim}_{\text{loc}}(\mu,z) \leq \frac{\delta}{\log(m_2)} - \left(\frac{1}{\log(m_2)} - \frac{1}{\log(m_1)} \right) \underline{\delta_1}.
$$

So, if $\underline{\delta_1} = \overline{\delta_1}$ and $\delta = \underline{\delta_1} + \delta_2$ then $\mu$ is exact dimensional with respect to $d$.
\end{lemma}

\begin{proof}
The first inequality follows from the proof of a result of Marstrand (see \cite[Theorem 5.8]{ref13}), while the second one can be deduced from the proof of \cite[Theorem 2.11]{ref14}.
\end{proof}

\end{appendix}

\newpage
\thispagestyle{empty}
\renewcommand{\thepage}{}
\bibliography{biblio2}
\bibliographystyle{plain}

\end{document}